\documentclass[11pt, oneside,reqno]{amsart}   	
\usepackage{geometry}
\usepackage{float}
\allowdisplaybreaks
\usepackage{parskip}    		
\usepackage{graphicx}				
\usepackage{amssymb}
\usepackage{amsmath}
\usepackage[all]{xy}
\usepackage{mathtools}
\usepackage{tikz-cd}
\usepackage{enumitem}
\usepackage{tikz-cd}
\usepackage{appendix}
\usepackage[T1]{fontenc}
\usepackage{esint}

\usepackage[%
bookmarks=true,
colorlinks,
linkcolor=blue,
urlcolor=blue,
citecolor=blue,
plainpages=false,
pdfpagelabels,
final,
breaklinks=true\between
]{hyperref}
\usepackage{hyperref}
\usepackage[numbers,sort&compress]{natbib}

\newtheorem{thm}{Theorem}[section]
\newtheorem{lem}[thm]{Lemma}
\newtheorem{cor}[thm]{Corollary}
\newtheorem{prop}[thm]{Proposition}
\newtheorem{conj}[thm]{Conjecture}

\newtheorem{defn}[thm]{Definition}

\numberwithin{equation}{section}

\def\Pb{\ifmmode{\Bbb P}\else{$\Bbb P$}\fi}
\def\Z{\ifmmode{\Bbb Z}\else{$\Bbb Z$}\fi}
\def\C{\ifmmode{\Bbb C}\else{$\Bbb C$}\fi}
\def\R{\ifmmode{\Bbb R}\else{$\Bbb R$}\fi}
\def\S{\ifmmode{S^2}\else{$S^2$}\fi}

\def\supp{\operatorname{supp}}

\def\Ric{\operatorname{Ric}}

\def\exp{\operatorname{exp}}

\def\S{\cal S}

\newcommand{\eps}{\varepsilon}

\newcommand{\cH}{\mathcal{H}}

\newcommand{\cM}{\mathcal{M}}
\newcommand{\cN}{\mathcal{N}}

\newcommand{\pr}{\partial}

\DeclareMathOperator{\area}{Area}

\begin{document}


\title[Fate of flow]{On the long-time limit of the mean curvature flow in closed manifolds} 
\begin{abstract} In this article, we show that generally almost regular flows, introduced by Bamler and Kleiner, in closed $3$--manifolds will either go extinct in finite time or flow to a collection of smooth embedded minimal surfaces, possibly with multiplicity. Using a perturbative argument then we construct piecewise almost regular flows which either go extinct in finite time or flow to a stable minimal surface, possibly with multiplicity. We apply these results to construct minimal surfaces in $3$--manifolds in a variety of circumstances, mainly novel from the point of the view that the arguments are via parabolic methods. \end{abstract}
\author{Alexander Mramor and Ao Sun}
\address{Department of Mathematics, University of Copenhagen, Universitetsparken 5, DK-2100 Copenhagen Ø, Denmark 2300
\newline
\newline
\indent Department of Mathematics, Lehigh University, Chandler--Ullmann Hall, Bethlehem, PA 18015
\newline}
\email{almr@math.ku.dk, aos223@lehigh.edu}
\maketitle

\section{Introduction}

The mean curvature flow, as the gradient flow of the area functional, is a natural tool to find minimal surfaces -- in particular area minimizing and stable ones. The $1$-dimensional mean curvature flow in surfaces, known as the curve shortening flow, has been proven to be a very powerful tool for constructing geodesics. An especially important accomplishment was by Grayson \cite{Grayson89_CSF}, who used the curve shortening flow to settle Lusternik and Schnirelman's proposal to construct $3$ distinct closed embedded geodesics in $(S^2,g)$.

A central obstruction to using $n$--dimensional mean curvature flow in $(n+1)$--dimensional ambient manifold to construct minimal surfaces is the more complicated nature of singularity formation in dimensions $n \geq 2$. Therefore, weak notions of the flow are necessary. There are a number of reasonable definitions of weak flows to handle this issue; these are discussed in section \ref{prelim}. Very recently, Bamler and Kleiner \cite{BamlerKleiner23_Multiplicity1} introduced the notion of almost regular flow, which are well-behaved Brakke flows whose so--called local scale functions satisfy a certain integrability condition. They proved that in $\R^3$ all generic flows (i.e., flows which only develop mean convex singularities) and hence, by an approximation argument, inner and outermost flows emanating from an embedded closed surface are almost regular. In this framework, they then proved the Multiplicity One Conjecture of Ilmanen \cite{Ilmanen95_Sing2D}. In this article, we use the notions of generic and almost regular flows in general $3$--manifolds and first show the following:
\begin{thm}\label{longterm}\label{smoothfate} Let $M$ be a $2$-sided properly embedded surface in a closed Riemannian $3$--manifold $(N,g)$, and let $M_t$ be an almost regular flow from $M$. Then either $M_t$ goes extinct in finite time or there exists a sequence $t_i \to \infty$ for which $M_{t_i}$ converges to $m_1\Sigma_1+\cdots+m_k\Sigma_k$ in the sense of varifolds, where each $\Sigma_j$ is a smoothly embedded minimal surface and each $m_j$ is a positive integer. Moreover, all the $\Sigma_j$'s are disjoint.
\end{thm}

We point out that some of the $\Sigma_j$ may be nonorientable, even if the initial data is $2$--sided: consider for instance a situation where the boundary of the tubular neighborhood of an embedded projective plane $P$ (say in the 3 manifold $N = \R P^3$), homeomorphic to $S^2$, flows back to $P$ as $t \to \infty$. This can also be interpreted as saying that the topology of the flow may actually "increase" in the limit; we discuss more precisely to what extent the topology of the $\Sigma_i$ can be controlled in terms of the initial data $M$ in Proposition \ref{limit_topology} below. We note that a priori the limit in Theorem \ref{longterm} may not be unique; if the ambient manifold is analytic and the convergence is with multiplicity one, we may appeal to the \L{}ojasiewicz-Simon inequality to see this is the case by \cite{Simon83Asym}, but we assume/show neither in the statement above. However, we are able to show that the limit stable minimal surface is unique whenever it is strictly stable. Recall below that a stable minimal surface is strictly stable if the first eigenvalue of the linearized operator is positive, and as pointed out, for example, in \cite{Urbano13_one-sided}, this notion continues to make sense even for $1$-sided minimal surfaces:
\begin{thm}\label{longterm_uniqueness}
    In the setting of Theorem \ref{longterm}, if the limit stable minimal surfaces are strictly stable, then any varifold converging limit of $M_{t}$ converges to this limit with the same multiplicities as $t\to\infty$. 
\end{thm}
See Theorem \ref{longterm_uniqueness1} below. We would like to point out that the uniqueness of the limit in geometric measure theory, in general, is quite hard when the multiplicity of convergence is greater than one.

Next, we further study what the possible limit minimal surfaces can be. Because the flow is the gradient of the area functional, one expects to typically find area minimizing or at least stable minimal surfaces using mean curvature flow. Nevertheless, in general, the long-time limit of mean curvature flow can be unstable. For curve shortening flow, a well-known example is the flow $\gamma_t$ of a closed embedded curve $\gamma \subset (S^2, g_{\text{round}})$ which bisects the sphere. It will converge, as $t \to \infty$, smoothly to an equator of the sphere; in contrast, curves $\gamma$ which do not bisect the sphere flow to round points. Most curves $\gamma$ do not bisect the sphere, and we can always perturb one that bisects the sphere to one that does not. Therefore, most curve shortening flows cannot converge to an unstable geodesic. This observation motivates us to consider the extent to which we can show the flow will generically avoid unstable minimal surfaces.

The main tool we use to pursue this is the avoidance principle. In $\R^{n+1}$, if two mean curvature flows of closed hypersurfaces are disjoint at time $0$, then they will remain disjoint for all future time. Moreover, the distance between these two flows is increasing. This basic property motivates the definition of weak set flows by Ilmanen \cite{Ilmanen92_LSF} and White \cite{White95_WSF_Top}. The avoidance principle in a manifold is more subtle, because in general two disjoint mean curvature flows can actually get closer -- consider, for instance, the flow near an area minimizing minimal surface. In a manifold equipped with a positive Ricci curvature metric, the avoidance principle is very strong by White \cite{White24_Avoidance}, and lets us easily show the following:  
\begin{thm}\label{Rc_thm} In a Ricci positive manifold $(N^3,g)$, the mean curvature flow of a generic $2$--sided embedded hypersurface $M$ will go extinct in finite time. 
\end{thm}

This can be viewed as a generalization of the generic behavior of the curve shortening flows in $S^2$. As a corollary, this yields another proof of the classical fact that $H_{2}(N,\R)$ vanishes for $3$--manifolds $N$ of positive Ricci curvature, see Corollary \ref{vanish}. In the case that $(N,g)$ is not Ricci positive there could be stable minimal surface, and we cannot expect the flows starting from generic embedded hypersurfaces to go extinct in finite time. However, we can use the avoidance principle to show that after perturbations, the flow can avoid those unstable limits. Our next theorem, inspired by the work of Colding and Minicozzi \cite{ColdingMinicozzi12_generic}, says that in closed $3$--manifolds, from any embedded initial data, one may construct piecewise generic mean curvature flows which avoid unstable minimal surfaces as limits:
\begin{thm} \label{pw_thm} Let $M^2$ be a $2$--sided properly embedded surface of a closed compact Riemannian $3$--manifold $(N^3,g)$. Then there exists a piecewise almost regular flow emanating from it which either goes extinct in finite time, or there exist $t_i \to \infty$, such that $M_{t_i}$ converges possibly with multiplicity to a (potentially disconnected) stable minimal surface. 
\end{thm} 

When the ambient metric is bumpy, which is a generic condition by work of White \cite{White91_Bumpy, White17_Bumpy}, we observe that the limit in the statement above is unique by Theorem \ref{longterm_uniqueness}. Note that in light of \cite{ChenSun24_Multi2inManifold}, there are examples where the flow asymptotically converges to a stable limiting surface with multiplicity greater than one, so higher multiplicity convergence in the above may actually occur, and in the case the limit is strictly stable seems to be robust under further perturbations. We will give the definition of piecewise almost regular flow in the section below, but roughly speaking, it is a concatenation of almost regular flows where, by hand, isotopies are performed at the jumps; the number of perturbations by isotopy taken will be finite, and the perturbations can be taken as small as we wish. In that sense, the statement above is a statement on generic flows following the convention of Colding and Minicozzi \cite{ColdingMinicozzi12_generic}.

Of course, it is certainly desirable to only require perturbing the initial data. The obstruction to doing so in the statement above essentially boils down to the possibility of convergence to an unstable minimal surface with multiplicity; in analogy to singularity analysis and the resolution of the Multiplicity One Conjecture by Bamler and Kleiner \cite{BamlerKleiner23_Multiplicity1}, one might expect this to not occur; see the analogous situation in Almgren-Pitts minmax theory \cite{Zhou19_multi1}. We recall the following conjecture in \cite{ChenSun24_Multi2inManifold}:

\begin{conj}[Conjecture 1.3 in \cite{ChenSun24_Multi2inManifold}]\label{conj_multi1}
    Suppose the long-time limit of an almost regular mean curvature flow in $(N^n,g)$ is an unstable minimal surface, then the convergence must be multiplicity $1$.
\end{conj}

Despite this, we point out that as applications, we get flow proofs of the following existence theorems for stable minimal surfaces. In the proofs of these, we will often specifically use piecewise generic flows, because the topological change through singularities is easy to understand. We also recall that obtaining stability of the minimal surfaces is significant, because these often have better--controlled topology and geometry, for instance, in the context of positive scalar curvature. These corollaries are not new, and there have been long--standing proofs of them using geometric measure theory or harmonic map theory, see \cite{FedererFleming60_Current, SchoenYau79_ExistenceIncompressibleMinSurf, SacksUhlenbeck81_minimal2sphere, MeeksYau80_Top3d,MeeksSimonYau82_MinSurf, FreedmanHassScott83_LeastArea}, among others. On the other hand, the construction of these minimal surfaces via flows, along with a number of fairly explicit perturbations, is quite amenable to numerical approximation with the caveat that one has found appropriate initial data to start with. In the following statements $(N,g)$ is a compact orientable Riemannian $3$--manifold, where in this case oriented surfaces are $2$--sided and vice versa:
\begin{cor}\label{existence1} Suppose $\pi_2(N)$ is nontrivial. Then $N$ contains an embedded stable minimal sphere or projective plane. 
\end{cor} 
Note that in the statement above, we invoke the sphere theorem to find embedded initial data for which to start the flow, and so the above does not constitute a new proof of it. As we pointed out already, compared to other techniques, a pitfall of using the mean curvature flow to find minimal surfaces is that it is best to start with well-prepared (i.e., embedded) initial data. In some cases, the flow can produce stable minimal surfaces of nontrivial topology, the most basic of which is the following. 

\begin{cor}\label{existence2} Suppose $N$ has a trivial second homotopy group and contains a homologically nontrivial embedded torus. Then $N$ contains an embedded stable minimal surface homeomorphic either to a torus or a Klein bottle.
\end{cor}
Recall that hyperbolic manifolds are aspherical, so the space of such $3$--manifolds with $\pi_2$ trivial is quite large and is even generic amongst $3$--manifolds from a number of perspectives, for example, see \cite{DunfieldTHurston06_Random3Mfd, Maher10_RandomHeegaard}. For surfaces of more general topology, we can show the following result, which is very much in the spirit of the existence result of Freedman, Hass, and Scott \cite{FreedmanHassScott83_LeastArea} -- in fact, it implies the result above as we discuss in Section \ref{applications}. The surfaces in the Corollary below are known as algebraically incompressible surfaces, and are incompressible in the more geometric sense via Dehn's lemma: 
\begin{cor}\label{existence3} Suppose $\iota: \Sigma \to N$ is an orientable properly embedded surface of $N$ of genus $g \geq 1$ for which $\iota_{*} \pi_1(\Sigma) \to \pi_1(N)$ is injective, and that $N$ has trivial second homotopy group. Then either $M = \iota(\Sigma)$ is homotopic to an orientable stable minimal surface or a double cover of a stable minimally embedded connect sum of $g+1$ projective planes, where $g$ is the genus of $\Sigma$. 
\end{cor}

By double cover of an embedded surface $P \subset N$, here we mean specifically a surface which is the boundary of a tubular neighborhood of $P$ in $N$. It is a fact (see, for instance, Lemma 3.6 in \cite{Hatcher07_Notes3Mfd}) that every class of $H_2(N,\Z)$ can be represented by an orientable properly embedded surface $M$. Recalling the Thurston norm $||\sigma||_{T}$ of $\sigma\in H_2(N,\Z)$ is the smallest attained genus of embedded surfaces homologous to $\sigma$. We see any such surface $M$ realizing the norm must be algebraically incompressible. If $||\sigma||_{T} \geq 1$ and if we suppose that $\pi_2(N)$ is trivial, we can apply the Corollary above, and claim that the minimal surface we obtain realizes the Thurston norm. We can record the discussion above as follows:
\begin{cor}\label{existence0}
   Suppose $\sigma\in H_2(N, \Z)$, with $\pi_2(N)$ trivial. Then either there exists an embedded oriented stable minimal surface $\Sigma$ in the class of $\sigma$ that realizes the Thurston norm of $\sigma$ or there is an element realizing the Thurston norm which is the double cover of an embedded nonoriented stable minimal surface. 
\end{cor}

Finally, we note that Theorem \ref{longterm} lets us substitute the tightening process in minmax for mean curvature flow, at least in some cases. A major issue in applying the flow to a sweepout is that the flow of a sweepout may not remain a sweepout, but this can be compensated for if the flows are not nonfattening. In particular, we show the following:
\begin{cor}\label{existence5}  Suppose $N$ is diffeomorphic to $S^3$. Then $N$ contains a minimal embedded $S^2$. 
\end{cor}
Originally, this result was proved by Simon and Smith using minmax theory, see also \cite{ColdingDeLellis03_minmax}. As to the actual best result known, we point out that recently Haslhofer and Ketover in \cite{HaslhoferKetover19_Min2Sphere} showed there are at least two minimal spheres in $S^3$ with a bumpy metric, and very recently Zhichao Wang and Xin Zhou in \cite{WangZhou23_4Sphere} proved that there are at least $4$ embedded minimal spheres in $S^3$ equipped with either a bumpy or Ricci positive metric. We expect that if Conjecture \ref{conj_multi1} is true, then the mean curvature flow can also be used to recover their results. We emphasize that in the above, we are not applying Theorem \ref{pw_thm} and the minimal surface above is potentially unstable, as must be the case when, for instance, $N$ is Ricci positive.

We conclude the introduction with a remark on the compactness assumptions made throughout. When the ambient manifold is not closed, but merely complete with uniform bounded curvature and injectivity radius lower bound, Theorems \ref{smoothfate} and \ref{pw_thm} remain valid, if we include the possibility that the flow can escape from any compact set. For example, if $N$ is $S^2\times\R$ with an infinite long trumpet shape (i.e. a cusp) in one direction, the flow of an appropriate $S^2$ section of it will escape to spatial infinity/clear out, neither terminating nor converging to a minimal surface. It seems to be a more complicated manner yet to generalize our results to the case $M$ is allowed to be noncompact (say properly embedded in a noncompact $N$), one very basic reason being that the area can be infinite. As before, the flow may clear out in the manner of the example above, but it seems that in some cases, if the flow doesn't clear out or go extinct, one could show it must flow to a minimal surface. For a basic example, a well known consequence of Ecker and Huisken \cite{EckerHuisken89_EntireGraph} says that uniformly Lipschitz graphs of bounded height in $\R^n$ must converge to flat, hence minimal, planes.

\subsection*{Acknowledgments} A.M. thanks Alec Payne and Felix Schulze for their interest and helpful comments. In the course of preparing the article he was supported by CPH-GEOTOP-DNRF151 from the Danish National Research Foundation and CF21-0680 from the Carlsberg Foundation (via GeoTop and N.M. M{\o}ller respectively) and is grateful for their assistance. A.S. is supported by the AMS-Simons travel grant. He is grateful to the Copenhagen Centre for Geometry \& Topology (GeoTop) at the
University of Copenhagen, especially N.M. M{\o}ller, for the invitation to visit in Spring 2024 and participate in the meeting "Masterclass:
Recent Progress on Singularity Analysis and Applications of the Mean Curvature Flow;" part of this work was initiated during the visit. The authors also thank the anonymous referees for their careful reading and helpful feedback, which helped improve the exposition of the article.

\section{Background and preliminary lemmas}\label{prelim}

Let $N^{n+1}$ be a Riemannian manifold and $X: M \to N^{n+1}$ be an embedding of a manifold $M^n$ realizing it as a $2$--sided smooth closed hypersurface of $N$, whose image by abuse of notation we also refer to as $M$. Then the mean curvature flow $M_t$ of $M$ is given by (the image of) $X: M \times [0,T) \to N^{n+1}$ satisfying the following: 
\begin{equation}\label{MCF_equation}
\frac{dX}{dt} = \vec{H}, \text{ } X(M, 0) = X(M)
\end{equation}

An immediate consequence of the first variation formula and the evolution equation above is that the mean curvature flow is the gradient flow of the area functional, in the sense that $\frac{d}{dt}\text{Area}(M_t) = - \int_{M_t} |\vec H|^2 d\mu_t$, where $\mu_t$ is the area measure; from this one sees the mean curvature is well adapted to finding minimal surfaces, which are critical points of the area functional.

The first obstruction to the construction of minimal surfaces using mean curvature flow is singularity formation: as a highly nonlinear system, mean curvature flow can develop finite time singularities; then it is even unclear how to continue the flow after the singular time. To overcome this issue, there are three major weak flow approaches: 
\begin{itemize} 
\item The Brakke flow.
\item The weak set flow and level set flow.
\item The flow with surgery.
\end{itemize}

The \textbf{Brakke flow}, introduced by Brakke in his thesis \cite{Brakke78}, is a geometric measure-theoretic definition in terms of (typically, integral) varifolds that satisfy the eponymous Brakke inequality. Here we present a definition essentially due to Ilmanen \cite{Ilmanen94_EllipReg}, and the expository is from \cite{BamlerKleiner23_Multiplicity1}. An $n$-dimensional Brakke flow defined on the time-interval $I\subset\R$ in a $(n+1)$-dimensional manifold $(N,g)$ is a family of Radon measures $(\mu_t)_{t\in I}$ such that:
\begin{itemize}
    \item For a.e. $t\in I$, the measure is integer $\cH^n$-rectifiable and the associated varifold has locally bounded first variation with variational vector $\vec H$ in $L^1$.
    \item For any compact set $K\subset N$, and any $[t_1,t_2]\subset I$, 
    \[
    \int_{t_1}^{t_2}\int_K |\vec H|^2 d\mu_t dt<+\infty.
    \]
    \item For any $[t_1,t_2]\subset I$ and $u\in C_c^1(N\times[t_1,t_2])$,
    \begin{equation}\label{eq:BrakkeF}
        \left.\int_N u(\cdot,t) d\mu_t \right|_{t=t_1}^{t=t_2}
        \leq \int_{t_1}^{t_2}\int_N (\pr_t u+\nabla u\cdot \vec H-u|\vec H|^2)d\mu_t dt.
    \end{equation}
\end{itemize}

One can quickly see that the Radon measures associated to a smooth/classical mean curvature flow, that is, a family of smooth manifolds satisfying \eqref{MCF_equation}, is a Brakke flow. A Brakke flow suffers from a serious problem: as \eqref{eq:BrakkeF} is only an inequality, the whole connected components of (the support of) a Brakke flow are allowed to disappear instantaneously! In the same monograph, Ilmanen \cite{Ilmanen94_EllipReg} (see also White \cite{White09_CurrentsVarifolds})
 gives a construction via elliptic regularization of a "well-behaved" Brakke flow out of smooth initial data for which these issues don't occur. Such a Brakke flow is known to be:
\begin{itemize}
    \item unit-regular: if the density of the Brakke flow at a spacetime point is $1$, then it is a smooth mean curvature flow in a spacetime neighborhood.
    
    \item cyclic mod $2$: suppose the unique associated rectifiable mod-$2$ flat chain of $\mu_t$ is $V(t)$, then $\pr[V(t)]=0$ for a.e. $t\in I$.
\end{itemize}

Even amongst such well-behaved Brakke flows, there happens to be an issue of uniqueness, which brings us to the level set flow which originated from the viscosity method \cite{EvansSpruck91, ChenGigaGoto91_LSF}. Later, Ilmanen \cite{Ilmanen92_LSF} and White \cite{White95_WSF_Top} introduced a purely set-theoretic definition using the avoidance principle. The classical avoidance principle says that two disjoint hypersurfaces in a manifold $N$ will stay disjoint, at least as long as one of them is compact. The \textbf{weak set flow} is defined to be a closed subset $\cM$ of $N\times\R$ such that if a smooth mean curvature flow does not intersect $\cM\cap\{t=t_0\}$, then it is disjoint from $\cM\cap \{t=t_1\}$ for all $t_1>t_0$. White \cite{White24_Avoidance} proved the following avoidance principle for weak set flows.

\begin{thm}\label{ambient_avoidance} Suppose that $N$ is a complete, connected Riemannian manifold with positive injectivity radius such that $|\nabla^k \text{Riem} |$ is bounded for each nonnegative integer $k$. Let $\Lambda$ be a lower bound for the Ricci curvature of $N$, and suppose that $t \in [0, \infty) \to X(t), Y(t)$ are weak set flows in $N$. Then 
\begin{equation}\label{dist}
e^{-\Lambda t} d(X(t), Y(t)) 
\end{equation} 
is an increasing function of $t$. 
\end{thm}

In particular, note that for $\Lambda > 0$ two flows don't merely stay disjoint but actively "repel" exponentially quickly -- this will play an important role in the proof of Theorem \ref{Rc_thm} below.

Indeed, the support of a Brakke flow is a weak set flow, as shown in Ilmanen's monograph. Although weak set flows may not be unique starting from a given hypersurface, there is a canonical weak set flow associated to any initial data called the \textbf{level set flow}: for a given set $A$ is the envelope of all families of sets initially agreeing with $A$ that satisfy the avoidance principle. As an aside, we point out that, in fact, the level set flow can even be defined this way for closed initial data that is not necessarily a regular hypersurface. While the existence and uniqueness are guaranteed, the level set flow of initial data may develop interior even if it is smooth, as first noted by Ilmanen and White \cite{White02ICM} (and recently appeared in \cite{IlmanenWhite24_Fattening}, see also related works \cite{AngenentIlmanenVelazquez17_fattening, LeeZhao24_MCFconical, Ketover24_Fattening}) -- this fattening can be seen as an indication of nonuniqueness for level set flows. To study the level set flow even in the presence of fattening Hershkovits and White \cite{HershkovitsWhite20_Nonfattening} introduced the notions of innermost/outermost flows. Given a closed hypersurface $\Sigma \subset \R^3$ that is the boundary of a domain $\Omega$, we can define:
\begin{itemize}
    \item The innermost flow $\cM^-$ is the boundary of the level set flow in $N\times\R$ generated by $\Omega\times\{0\}$;
    \item The outermost flow $\cM^+$ is the boundary of the level set flow in $N\times\R$ generated by $(N\backslash\Omega)\times\{0\}$.
\end{itemize}
Hershkovits and White noticed that if the level set flow $\cM$ starting from $\Sigma$ fattens, then $\cM^-$ and $\cM^+$ are different. It is expected that if the level set flow fattens, the innermost/outermost flows are natural candidates to serve as the canonical representatives of the weak set flow.

It is not immediate to define inner and outermost flows in general $3$--manifolds from a $2$--sided surface $M$ because such a surface need not be separating, as a cross section of $S^2 \times S^1$ illustrates. In the sequel we will define such flows as the following: for a choice of normal vector $\nu$ consider the domains $\cM^-$ and $\cM^+$ bounded by $M$ and $M \pm \epsilon \nu$ respectively, for some small $\epsilon > 0$, where $M \pm \epsilon \nu$ are surfaces that are generated by pushing $M$ in $\epsilon\nu$ direction by the exponential map. Then we define the inner and outermost flows to be the boundary component of the level set flows of $\cM^-$,  $\cM^+$, respectively, emanating from $M$.

The third weak flow is the \textbf{mean curvature flow with surgery}, where one preemptively cuts out regions developing high curvature, classifies them topologically, and then continues the flow in a piecewise manner, see \cite{Head13_MCF2convex, Lauer13_MCFSurgery, BrendleHuisken16_MCFSurgeryR3, HaslhoferKleiner17_MCFsurgery, Daniels-Holgate22_SurgeryApprox}. Locally, the high curvature regions will be modeled on round spheres or cylinders in coordinate patches, which will bound solid balls/cylinders, respectively. Regions where high curvature transitions into low curvature (i.e. "neckpinches") the surface are modeled on cylinders and at these points one "cuts" the neck at an appropriate spot and places in caps/discs of controlled geometry, in particular so that the remaining low curvature regions have curvature upper bounded by a controlled constant (which is amongst the analytical difficulties in establishing the flow with surgery). We note that the high curvature regions, after cap placement, will be homeomorphic to either $S^n$ or $S^{n-1} \times S^1$, which makes the flow with surgery a useful tool in topology. This particular consequence won't be so relevant in this work, but the well controlled, concretely understandable change of topology across surgeries will be helpful in the sequel. 

\subsection{Almost regular flow and generic approximation}
While these notions of weak flows seem quite different, in $3$--manifolds there is a relatively complete picture to relate all of them. In \cite{BamlerKleiner23_Multiplicity1}, Bamler and Kleiner introduced a new notion of weak flow called \textbf{almost regular flow}, which are well behaved Brakke flows in the sense described above, along with an extra integrability condition on what they call scale functions, which measure the size of neighborhoods in which the flow is smooth. We refer the readers to \cite{BamlerKleiner23_Multiplicity1} for the definition and properties of almost regular flows. Here, we only state properties of almost regular flows that we reference:

\begin{enumerate}
\item Generic flows, i.e., ones which only encounter mean convex singularities (see discussion and references below), are almost regular. 
\item The innermost/outermost flows that start from a closed smooth embedded surface can be approximated by generic flows, and as a consequence, are almost regular flows for $t\geq 0$. In particular, if the level set flow from a closed smooth embedded surface is nonfattening, then it is almost regular.
    \item An almost regular flow is regular for a.e. $t\in I$.
    \item (Lemma 2.7 in \cite{BamlerKleiner23_Multiplicity1}) If $\cM$ is almost regular, then for any test function $u\in C_c^1(N\times[t_1,t_2])$ where $[t_1,t_2]\subset \text{Int} I$, $t\to\int_N u(\cdot,t)d\mu_t$ is continuous and the equality holds in \eqref{eq:BrakkeF}. 
\end{enumerate}

While the almost regular flows have many nice properties, they still possibly develop complicated singularities that obscure the possible topological change of the flow through them. On the other hand, as we have described before, the mean curvature flow with surgery has a clear description of the topological change of the flow. Therefore, we would like to use mean curvature flow with surgery to approximate the almost regular flows we use in the sequel, particularly in showing (some of) the applications mentioned in the introduction.

In order to use mean curvature flow with surgery to approximate the almost regular flows we use in practice, we first use a well behaved (unit regular, cyclic) Brakke flow with only cylindrical and spherical singularities to approximate them. Such flows we call generic. It has been conjectured by Huisken and Angenent-Ilmanen-Chopp that mean curvature flow with generic initial data can only have cylindrical and spherical singularities, and recently, there has been various progress on this conjecture, \cite{ColdingMinicozzi12_generic, CCMS20_GenericMCF, chodoshchoischulze2023mean, SunXue2021_initial_closed, SunXue2021_initial_conical}. Using the result of Chodosh, Choi, Mantoulidis, and Schulze \cite{CCMS20_GenericMCF} and Chodosh, Choi, and Schulze \cite{chodoshchoischulze2023mean}, Bamler and Kleiner \cite{BamlerKleiner23_Multiplicity1} proved that a mean curvature flow starting from a generic closed embedded surface in a three-manifold can only have cylindrical and spherical singularities. On the other hand, Daniels-Holgate \cite{Daniels-Holgate22_SurgeryApprox} proved that the mean curvature flows with surgeries can approximate a mean curvature flow with only cylindrical and spherical singularities, where the approximation is in the Hausdorff distance. In addition, this approximation is smooth away from the singular set. Combining all the ingredients above gives the following approximation result in $\R^3$:

\begin{prop}
\label{goodflow_ambient}
    The inner and outermost flows starting from a closed embedded surface in $\R^3$ can be approximated by mean curvature flows with surgeries in Hausdorff distance, and the approximation is smooth away from the singular set.
\end{prop}
Note that in the classical case, i.e., $N = \R^{3}$, properly embedded hypersurfaces are automatically orientable. If $N$ is more generally simply connected, it is not hard to see this true by an intersection number argument, but we point out that it does not suffice if $N$ is merely orientable with as the example $\R P^2 \subset \R P^3$ shows. With this in mind, we claim that the following holds true in general closed $3$--manifolds: 
\begin{prop}
    Let $M$ be a closed, properly embedded $2$--sided surface in a closed $3$--manifold $N$. Then the inner and outermost flows of $M$, as defined above, can be approximated by mean curvature flows with surgeries in Hausdorff distance over any fixed time interval, and the approximation is smooth away from the singular set.
\end{prop}
\begin{proof} The main point is to establish the existence of generic mean curvature flow ala \cite{CCMS20_GenericMCF, chodoshchoischulze2023mean}, and its subsequent approximation by surgery flows, then the statement about inner and outermost flows follows exactly as in the Euclidean case by an approximation argument as discussed in Section 7 of \cite{BamlerKleiner23_Multiplicity1}. As is well known, in a blowup limit, singularities will still be modeled on solitons in $\R^3$, so we mostly just supply a sketch of this argument. Noting that because $M$ is $2$--sided, one may meaningfully discuss one-sided perturbations of $M$, and the point made in the previous sentence, one readily sees from \cite{CCMS20_GenericMCF, chodoshchoischulze2023mean} that if one encounters a nongeneric singularity of multiplicity one that the initial data can be perturbed to avoid the singularity at that spacetime point. Because $M$ is a compact surface it has finite genus and so by genus monotonicity (see Section 11 of \cite{CCMS20_GenericMCF}) and the classification of genus $0$ shrinkers by Brendle \cite{Brendle16_genus0}, only finitely many (arbitrarily small) perturbations of the initial data are required to produce a generic flow up at least to the first time a high multiplicity tangent flow develops. Now all these generic flows, as shown in Section 7 of \cite{BamlerKleiner23_Multiplicity1}, are almost regular flows; note this uses the resolution of the mean convex neighborhood conjecture by Choi, Haslhofer, and Hershkovits \cite{ChoiHaslhoferHershkovits18_MeanConvNeighb}, which the authors note carries over in the curved ambient setting. Because under rescaling as in singularity analysis the ambient metric converges smoothly to the Euclidean metric, one can see by contradiction that the resolution of the multiplicity one conjecture applies in the curved ambient setting because all the relevant geometric quantities can be arranged to evolve as they do in the flat case up to arbitrarily small errors -- their argument at a high level is by showing in the case of high multiplicity convergence in a tangent flow that a certain quantity satisfies contradictory bounds so this suffices for our needs. As a result we get that for generic choice of data there will be a flow along which only mean convex singularities may develop for all time; by the resolution of the mean convex neighborhood conjecture \cite{ChoiHaslhoferHershkovits18_MeanConvNeighb} and the barrier argument of Hershkovits and White \cite{HershkovitsWhite20_Nonfattening, HershkovitsWhite23_Avoid}, it is unique among almost regular flows starting from $M$. The approximation theorem of Daniels-Holgate also applies in this setting; most importantly, besides the mean convex neighborhood conjecture, the required local estimates for surgery also hold in the general ambient setting as shown by Haslhofer and Ketover \cite{HaslhoferKetover19_Min2Sphere}. 
\end{proof}

We next present some properties of such almost regular flows, where $M$ is, as usual, a properly embedded $2$--sided surface in a given closed $3$--manifold $N$. 

\begin{lem}\label{lem:PreserveOriented}
Suppose that $M_t$ is a generic or inner/outermost flow out of $M$. Then $M_t$ will also be $2$--sided for all regular times $t$.
\end{lem} 
\begin{proof} It's obvious if $M_t$ is smooth then it will remain $2$--sided, because the flow is an isotopy. In the case that $M_t$ has singularities, we first observe that it is easy to see for the approximating smooth flows $S_t$ to $M_t$ produced in \cite{Daniels-Holgate22_SurgeryApprox} that $S_t$ will remain orientable if $M$ is. The convergence of the approximating surgery flows is smooth during regular times, giving the claim. 
\end{proof}

An important subtlety to remember, though, as mentioned in the introduction, is that orientability is not necessarily inherited by the limit we discuss in the next section. In the next statement, we write $M\sim \sigma$ to indicate that $M$ is homologous to $\sigma$. 

\begin{lem}\label{lem:PreserveHomology}
 Suppose that $M_t$ is a generic or inner/outermost flow out of $M$, with $M\sim \sigma$ where $\sigma\in H_2(N,\Z)$. Then  $M_t \sim \sigma$ for all regular times $t\geq 0$.
\end{lem}

\begin{proof}
Again, using the approximation by surgery flows is smooth during regular times, one only needs to note that surgery does not change the $2$nd homology class of a regular surface. 
\end{proof}

\begin{lem}\label{lem:GenusNonincreasing}
Suppose that $M_t$ is a generic or inner/outermost flow out of $M$.  Then the genus of $M_t$, considered at smooth times, is nonincreasing along the flow. 
\end{lem}

\begin{proof}
As in the previous proofs, it suffices to note this is the case for approximating surgery flows. Now, in the surgery, components are either deleted or cylindrical regions are cut and replaced with caps, and neither of these actions do not increases the genus, giving the claim. 
\end{proof}

We show in the next section that an almost regular flow will either go extinct or flow to a potentially unstable smooth minimal surface, possibly with multiplicity. If the minimal surface is unstable, we will afterwards wish to perturb the flow to ensure it doesn't flow to it, but as discussed in the introduction, for technical reasons, it seems difficult to carry this out by just perturbing the initial data if the multiplicity is high (i.e. greater than or equal to $2$). With this in mind, we will consider the following piecewise almost regular flows in the sequel:
\begin{defn}\label{pwflow_def} Let $M$ be a properly embedded, $2$--sided surface of a $3$--manifold $N$. Then a piecewise almost regular flow $M_t$ emanating from $M$ defined on $[0, \infty)$ is given by an increasing (possibly infinite) sequence of times $t_i \to \infty$ for which 
\begin{enumerate}
\item $M_t$ on $[0, t_1]$ is an almost regular flow out of $M$ satisfying the conclusions of Proposition \ref{goodflow_ambient} above and so that $M_{t_1}$ is a smooth surface.
\item Denoting by $M_t^1$ on $[0, t_1)$ the flow from above, we let $\tilde{M}^1_{t_1}$ be a surface that is isotopic to $M^1_{t_1}$, and we define $M^2_t$ out of this perturbed surface $\tilde{M}^1_{t_1}$ as above. 
\item Similarly, we define $M^i_t$ for $i \geq 2$ inductively. 
\end{enumerate}
\end{defn}

We end this section with the following technical tool to ensure the limit of the flow we take is smooth:

\subsection{Ilmanen's regularity theorem}
Because the mean curvature flow is the gradient flow of the area functional, the long-time limit is expected to be a critical point of the area functional. From geometric measure theory, the limit is known to be a stationary varifold, and without any other assumptions, the limit can have singularities.

In \cite[Section 4, proof of Theorem 1.1]{Ilmanen95_Sing2D}, Ilmanen used Simon's sheeting theorem \cite{Simon94_Willmore} to show that the tangent flow of a first-time singularity of a mean curvature flow of closed embedded surfaces in $\R^3$ must support on a smooth embedded shrinker, possibly with multiplicity. Ilmanen's argument is very general and can be adapted to study the long-time limit of mean curvature flow. For the sake of clarity, we precisely state the version that we will use in this paper:

\begin{thm}[Ilmanen's regularity for limit surfaces, \cite{Ilmanen95_Sing2D}, Section 4, proof of Theorem 1.1; see also \cite{BamlerKleiner23_Multiplicity1}, Lemma 2.13]\label{thm:IlReg}
    Suppose $\{M_i\}_{i\in\Z_+}$ is a sequence of closed embedded surfaces in a closed manifold $(N^3,g)$, satisfying the following assumptions:
    \begin{enumerate}
        \item $\area(M_i)$ and $\text{genus}(M_i)$ are uniformly bounded from above;
        \item $\int_{M_i}|\vec H|^2 d\mu_{M_i}\to 0$ as $i\to\infty$;
        \item (Area growth bound) there exists $r_0>0$ such that there exists a constant $C$ such that for any $x\in N$ and $r\in(0,r_0)$, $\area(M_i\cap B_r(x))\leq Cr^2$. Here $B_r(x)$ can either be the intrinsic ball in $(N,g)$ or the extrinsic ball for a fixed isometric embedding $(N,g)\to (\R^m,g_{\text{Euc}})$.
    \end{enumerate}
    Then there exists a subsequence of $\{M_i\}_{i\in\Z_+}$, still denoted by $\{M_i\}$, that converges to an integral varifold $V$ in the sense of varifold, such that $V$ is supported on a finitely many disjoint closed embedded minimal surface $\Sigma_1,\cdots,\Sigma_k$, with multiplicity $m_1,\cdots,m_k\geq 1$. Moreover, for any $j=1,2,\cdots,k$, there exists finitely many points $p_{j,1},\cdots,p_{j,\ell_j}$ inside $\Sigma_j$, such that for any $r>0$ that is sufficiently small, for sufficiently large $i$, $M_i\backslash \bigcup_{l=1}^{\ell_j} B_{r}(p_{j,l})$ is the union of $m_k$ number of disjoint connected components, and each component is isotopic to and converges in the varifold sense with either multiplicity $1$ or $2$ to $\Sigma_j \backslash \bigcup_{l=1}^{\ell_j} B_{r}(p_{j,l})$ depending on if the limit is two sided or not. 

\end{thm}

\section{Long-time behavior of almost regular flows: the proof of Theorem \ref{longterm}}

In this section, we show that in the long term, an almost regular flow either goes extinct in finite time or limits to a collection of smooth, but possibly unstable, minimal surfaces. Note that almost regular flows are not necessarily generic, although in further applications we will often suppose our flow is -- for posterities sake, we work in this more general class as much as possible. We finish with some properties of the minimal surfaces we find, should the limit be nonempty. 

Although as stated, the following lemma is not required to invoke Theorem \ref{thm:IlReg} below, the reader may recall that an important step in its proof in \cite{Ilmanen95_Sing2D} is to show that the surfaces involved have finite total curvature (see the proof of Theorem \ref{smoothfate} below). With this in mind, we start with the following total curvature bound for the sake of exposition: 

\begin{lem}\label{tabounds} Let $(N^3,g)$ be 
a closed manifold and $(M_t)_{t\in[0,\infty)}$ is an almost regular flow. For each $\epsilon > 0$, there exists a $C = C(\epsilon, M_0,g)$ such that $\int_{M_t} |A|^2\leq C$ for $t\in[0,\infty)$ away from a set of measure $\epsilon$.  
\end{lem}
\begin{proof}
Because $M_t$ is regular for a.e. time $t\geq 0$, we can apply the Gauss-Codazzi equations to get $R_N = K + 2 \Ric_N(\nu,\nu) + |A|^2 - H^2$, where $R_N$ and $\Ric_N$ are the scalar and Ricci curvatures on $N$ respectively. With this in hand, by Gauss--Bonnet we have for any $t > 0$ that is a regular time:
\begin{equation}
\int_{M_t} |A|^2 = \int_{M_t} H^2 - 4 \pi \chi(M_t) + C\text{Vol}(M_t).
\end{equation}
Where the constant $C$, independent of $t$, depends on the curvature of $N$. By the property of almost regular flow \cite{BamlerKleiner23_Multiplicity1}, the genus of $M_t$ is bounded by a uniform constant $G$, and hence $\chi(M_t)$ is also bounded by a constant. Of course $\text{Vol}(M_t)$ is bounded in terms of $\text{Vol}(M)$. Since the time derivative of the area is given by $- \int_{M_t} H^2$ (for weak flows, this is a consequence of Brakke's inequality), we see measure of the set of times for which $\int H^2$ is greater than a constant $C_1$ is bounded by $C_1/\text{area}(M_0)$. These two observations combined give what we want. 
\end{proof}

Similarly, we have the following:
\begin{lem}\label{decay} $\liminf\limits_{t\to\infty}\int_{M_{t}} H^2 = 0$.
\end{lem}
\begin{proof}
Supposing not, then there exists some $\epsilon > 0$ for which $\int_{M_t} H^2 > \epsilon$ for all sufficiently large $t$. This contradicts that $M$ is compact and so of finite area and that $\frac{d}{dt} \text{Area}(M_t) \leq -\int_{M_t} H^2$ by the Brakke inequality. 
\end{proof}

With Lemma \ref{decay} in hand, Allard's compactness theorem implies the following convergence theorem:

\begin{cor}\label{lem:VariConv}
There exists a sequence of time $t_i\to\infty$ such that $M_{t_i}$ converges to an integral stationary varifold $V$.
\end{cor}

Of course, from general theory, $\supp V$ is not necessarily a regular minimal surface. However, for surfaces in $3$--manifolds, we have more powerful tools to pick a convergence subsequence whose limit is regular, using Ilmanen's regularity theory, Theorem \ref{thm:IlReg}. We adopt his argument here; Lemma \ref{tabounds} gives us the total curvature bounds we require in the application of Simon's sheeting theorem, but there is a central feature of mean curvature flows in $\R^3$ that does not hold for mean curvature flows in a $3$--manifold, which requires us to take some extra care. Namely Huisken's monotonicity formula shows that a mean curvature flow $M_t$ of surfaces in $\R^3$ has a uniform area growth bound $\area(M_t\cap B_R(x))\leq D\pi R^2$ for all $R>0$, $x\in\R^3$ and $t>0$, where $D$ only depends on $M_0$; see Proposition 3.3 in \cite{sun-generic-multi-1}. This area growth bound is crucial in applying Simon's sheeting theorem in \cite{Ilmanen95_Sing2D}; in particular, see Theorem 10 and Corollary 11 of \cite{Ilmanen95_Sing2D}.

When $M_t$ is a mean curvature flow in a closed manifold $(N,g)$, by isometric embedding $(N,g)$ into $\R^m$, $M_t$ is a mean curvature flow with additional forces $\beta$, where $\beta$ is a tangent vector field on $N$, depending on the second fundamental form of $(N,g)$ in $\R^m$. Because of the forcing term, Huisken's monotonicity formula can not directly provide uniform area growth bound for all future time $t$, but luckily, a modified version of it can for short times, which will suffice for our purposes.

 To start, let us fix an isometric embedding of $(N,g)$ into $\R^m$;  recalling that we suppose $N$ is compact and denote by $L < \infty$ an upper diameter bound for $N$ under this embedding. Then it happens that the forcing term $\beta$, $\beta(x)\in T_xN$, only depends on the second fundamental form of the isometric embedding of $(N,g)$. Hence $\|\beta\|_{L^\infty}\leq C$ for some fixed constant $C$. Considering
\begin{equation}
    \rho_{x_0,t_0}(x,t)=(4\pi(t_0-t))^{-1}e^{-\frac{|x-x_0|^2}{4(t_0-t)}}
\end{equation}
which is the backward heat kernel in $2$-dimensions, the monotonicity formula of mean curvature flow with additional forces \cite{White97_Stratif, AM_CMgenericambient} shows that for $0\leq t_1\leq t_2<t_0$,
\begin{equation}\label{eq:MonoForce}
    \int_{M_{t_2}}\rho_{x_0,t_0} (x,t_2) dx 
\leq 
e^{C(t_2-t_1)}\int_{M_{t_1}}\rho_{x_0,t_0}(x,t_1)dx.
\end{equation}
We remark that the monotonicity formula also works for Brakke flows with additional forces, so this inequality holds even if $M_t$ is not necessarily smooth for $t\in[t_1,t_2]$. The following lemma shows that, while we may not get a uniform area growth bound for all $t>0$, we can show that for sufficiently large $t$, $M_t$ has a uniform area growth bound.

\begin{lem}\label{lem:AreaGrowth}
    There exists $D>0$ and $\mathbf{t}>0$ such that for $t>\mathbf{t}$, and any $R\in(0,L]$, $\tau\in(1,2)$ and $x_0\in \R^m$ that $\area(B_R(x_0)\cap M_{t_i+\tau})\leq D\pi R^2$ where $L$ is the diameter bound of $N$ as above. 
\end{lem}

\begin{proof}
    Because $\int_{M_t} \rho_{x_0,t_0}(x,t_0-R^2)\geq (4\pi R^2)^{-1}e^{-\frac{1}{4}}\area(M_t\cap B_R(x_0))$, we have  
    \[
    \area(M_t\cap B_R(x_0))\leq 4\pi R^2 e^{\frac{1}{4}}\int_{M_t} \rho_{x_0,t_0}(x,t_0-R^2)dx.
    \]
    By \eqref{eq:MonoForce}, it suffices to show that exists a fixed constant $\mathbf{t}>0$ such that 
    \[
    \int_{M_{t-\mathbf{t}}}\rho_{x_0,t+R^2}(x,t-\mathbf{t})dx
    \]
    has a uniform upper bound for all $t>\mathbf{t}$, $x_0\in N$ and $R\in(0,L]$. We observe that
    \[
    \begin{split}
        \int_{M_{t-\mathbf{t}}}\rho_{x_0,t+R^2}(x_0,t-\mathbf{t})
    =&
    \int_{M_{t-\mathbf{t}}}
    (4\pi(R^2+\mathbf{t}))^{-1} e^{-\frac{|x-x_0|^2}{4(R^2+\mathbf{t})}}dx
    \\
    \leq &
   (4\pi(R^2+\mathbf{t}))^{-1}\area(M_{t-\mathbf{t}})
    \leq (4\pi\mathbf{t})^{-1}\area(M_0).
    \end{split}
    \]
    Therefore, we can choose $\mathbf{t}=1$ to get a uniform upper bound, which yields the uniform area growth bound for $R\in(0,L]$.
\end{proof}
We are now ready to show the first result discussed in the introduction:

\begin{proof}[proof of Theorem \ref{smoothfate}] If the area of $M_t$ is sufficiently small, the clearing out lemma \cite[12.2 and 12.5]{Ilmanen94_EllipReg} shows that $M_t$ goes extinct (i.e. is supported on a set of zero measure) a short time later. Therefore, we may suppose its area is uniformly bounded below for all finite times. Considering a sequence $t_i \to \infty$ as in Lemma \ref{decay} above, by Allard compactness we may extract a converging subsequence which will be a stationary varifold $V$. By the almost all time regularity of the almost regular flows, we may suppose this set of times is regular as well. We also have the area growth estimate of $M_{t_i}$ from Lemma \ref{lem:AreaGrowth}, so that we may apply Theorem \ref{thm:IlReg}. For a brief overview, with the facts above in hand we may employ Simon's sheeting theorem ala Ilmanen \cite[Theorem 10 and Corollary 11]{Ilmanen95_Sing2D} along a sequence of smooth times $t_i \to \infty$ of $M_{t_i}$ to say that $V$ supports on a finite collection of smooth, embedded minimal surfaces $\{\Sigma_1,\cdots,\Sigma_k\}$ and the convergence is $m_j$--sheeted away from a finite number of "bridge points" towards $\Sigma_j$. By the regularity of the limit, and the embeddedness of $M_{t_i}$, all the $\Sigma_j$'s are disjoint via the maximum principle. 
\end{proof}

    Higher multiplicity convergence can indeed occur, see \cite{ChenSun24_Multi2inManifold}, and in the case one of the $\Sigma_i$ is strictly stable should be a robust phenomenon. Also, it is easy to see that the limiting set could be disconnected. Take, for instance, the initial data given by two disconnected area-minimizing closed embedded (say, strictly stable) minimal surfaces connected by a very thin neck. Then shortly after the flow is started, the neck will quickly pinch and retract to each minimizing surface, and the limit of the flow will be the union of these two disconnected minimal surfaces.

 It is interesting to ask to what extent the limit is unique.  If the convergence has multiplicity $1$, and the manifold is analytic, then by the classic \L{}ojasiewicz-Simon inequality, together with the Brakke/White regularity theory, $M_t$ indeed convergence to $\Sigma$ smoothly as $t\to\infty$, not only subsequentially. Without this machinery, even in the most well-known case, the curve shortening flow on surfaces, this question is not completely understood. In fact, in \cite{Grayson89_CSF} Grayson conjectured that there exists an example of curve shortening flow whose limit geodesics are not unique.

It is not so hard, though, to see that the limit will be unique if the limiting minimal surfaces are strictly stable, which we show next. To start, the following lemma shows that if we carefully choose $t_i$, $M_{t_i}$ will stay in a neighborhood of the limit surfaces. In other words, those necks described in the example above indeed must disappear in a very short period of time.

\begin{lem}\label{lem:LimNbhd}
     In the context of Theorem \ref{smoothfate} given a sequence $s_i \to \infty$ along which the flow converges to a collection of minimal surfaces $\Sigma_j$ for any $\delta>0$ we can choose another sequence of $t_i\to\infty$ such that, when $t_i$ is sufficiently large, $M_{t_i}$ must lie inside the $\delta$-tubular neighborhood of $\cup_{j=1}^k \Sigma_j$. 
\end{lem}

\begin{proof}
   Denote by $\cN_{\delta}(\Sigma)$ the (open) $\delta$-tubular neighborhood of a closed embedded surface $\Sigma$ in $N$, and consider a sequence of $s_i\to\infty$ as in Theorem \ref{smoothfate} such that $M_{s_i}$ converges to $m_1\Sigma_1+\cdots+m_k\Sigma_k$ in the sense of varifolds. For any $x_0\in N\backslash(\bigcup_{j=1}^k\cN_{\delta}(\Sigma_j))$, $\area(M_{s_i}\cap B_{\delta}(x_0))\to 0$ because $B_{\delta}(x_0)\cap(\bigcup_{j=1}^k\Sigma)=\emptyset$. In particular, this shows that for any $\eps>0$, $x'\in B_{\delta/2}(x_0)$, and $\eta\in(1/2,3/4)$, the Gaussian weight
    \[
    \int_{M_{s_i}}\rho_{x',s_i+(\eta\delta)^2}(x,s_i)dx<\eps, \text{when $s_i$ is sufficiently large.}
    \]
    Then by \eqref{eq:MonoForce} again, we see that when $\eps$ is chosen sufficiently small, the Gaussian density of $x'\in B_{\delta/2}(x_0)$ of $M_t$ for $t\in(s_i+\delta/2,s_i+3\delta/4)$ is strictly less than $1$, hence must be $0$. In other words, for $t\in(s_i+\delta/2,s_i+3\delta/4)$, $M_t$ is disjoint from $B_{\delta/2}(x_0)$ for sufficiently large $s_i$. Because $N\backslash \bigcup_{j=1}^k\cN_{\delta}(\Sigma_j)$ is compact, we can cover it by finitely many $B_{\delta/2}(x_0)$'s. This allows us to pick a sequence of $t_i\to\infty$ such that $M_{t_i}\subset \bigcup_{j=1}^k\cN_{\delta}(\Sigma_j)$ that have bounded total curvature and total mean curvature tending to zero. The argument of Theorem \ref{smoothfate} may then be rerun on this sequence to extract a subsequence which converges to a union of smoothly embedded minimal surfaces $\Sigma_j'$. Using these minimal surfaces, the conclusion follows by the triangle inequality. 
\end{proof}

We remark that in the above proof, conceivably $\Sigma_i \neq \Sigma_i'$, and in particular, the further chosen subsequence $M_{t_j}$ may not converge to the same collection of closed embedded minimal surfaces $m_1\Sigma_1+\cdots+m_k\Sigma_k$. But their limit must be very close to them, in the sense that the limit lies in a tubular neighborhood of these minimal surfaces. Now, recall that a stable minimal surface $\Sigma$ is strictly stable if the linearized operator $L_\Sigma$ has only positive eigenvalues. Using a barrier argument and the lemma above, we observe we may show that if the limit stable minimal surfaces are strictly stable, then the limit is unique, even if the convergence is possibly in high multiplicity and in a merely smooth (but not necessarily analytic) background manifold $N$:

\begin{thm}\label{longterm_uniqueness1}
    Suppose $M_t$ is an almost regular flow such that there exists $t_i\to\infty$ such that $M_{t_i} \to m_1\Sigma_1+\cdots+m_k\Sigma_k$ as in Theorem \ref{smoothfate}, and each $\Sigma_j$ is strictly stable. Then the limit is unique: namely, for any sequence of $s_i\to\infty$ so that $M_{s_i}$ converges as varifolds, the limit must be $m_1\Sigma_1+\cdots+m_k\Sigma_k$.
\end{thm}

\begin{proof}
    For any $\delta>0$, we can apply Lemma \ref{lem:LimNbhd} to claim that there is another sequence, which we still label $t_i$, such that $M_{t_i}$ lies inside the union of the $\delta$-tubular neighborhood of the $\Sigma_j$'s when $t_i$ is sufficiently large.

    To proceed, we use the neighborhood foliation structure of strictly stable minimal surfaces: because each $\Sigma_j$ is strictly stable, if it is $2$-sided, then there exists a tubular neighborhood $\cN^j$ of $\Sigma_j$ such that $\cN^j$ is foliated by surfaces $\Sigma^t$ for $t\in(-s,s)$, such that when $t\in(-s,0)\cup(0,s)$, $\Sigma^t$ has positive (inward pointing) mean curvature. For $1$-sided $\Sigma_j$, there is a similar structure of the neighborhood. For a proof of this fact, see \cite[Lemma A.1]{StevensSun24_LargeArea} for $2$-sided case, and \cite[Section 6]{StevensSun24_LargeArea} for $1$-sided case.

    Using this fact, we can find a tubular neighborhood $\cN^j$ of $\Sigma_j$ such that $\pr\cN^j$ has strictly positive mean curvature. In particular, if at the beginning we choose $\delta>0$ sufficiently small so that the $\delta$-tubular neighborhood is contained in $\cN^j$, then $M_t$ will lie within $\bigcup_{j=1}^k\cN^j$ for $t\geq t_i$. Hence, we see any limit of $M_t$ has to be supported in $\bigcup_{j=1}^k\cN^j$. If $\delta > 0$ is sufficiently small we may apply Brakke regularity to see the flows of boundary components of $\bigcup_{j=1}^k\cN^j$ are smooth for all $t > 0$ and asymptote back to $\bigcup \Sigma_j$, so by the comparison principle any converging limit of $M_t$ has to be supported on $\bigcup_{j=1}^k\Sigma_j$. The question of convergence in varifold topology to this entire set is more subtle, though, because in the above we may only apply Ilmanen's analysis to a special sequence of times $s_i \to \infty$.

 To deal with this, first we point out that for any sequence $s_i \to \infty$, the argument from the proof of Theorem \ref{smoothfate} does apply, we get the limit must be supported precisely on $\bigcup_{j=1}^k\Sigma_j$. Moreover, by the monotonicity of area along the mean curvature flow, the multiplicities may not vary either, so that for any sequence $t_i \to \infty$ for which $M_{t_i}$ converges "well" as in the proof of Theorem \ref{smoothfate}, $M_{t_i}\to m_1\Sigma_1+\cdots+m_k\Sigma_k$ as $t_i\to\infty$. Now, for a given $1 \leq j \leq k,$ choose a small $\epsilon > 0$, let $u_{j, \epsilon, D}$ be a smooth function which is supported on the $\epsilon$ tubular neighborhood of a domain $D \subset \Sigma_j$. Because almost regular flows are smooth for almost all times and continuous in the weak topology in time (see Lemma 2.7 in \cite{BamlerKleiner23_Multiplicity1}) Lemmas \ref{tabounds}, \ref{decay} give us that for any $\delta > 0$ we can approximate any sequence $s_i \to \infty$ by another sequence $s_i' \to \infty$ so that $|\int_{M_{s_i}} u_{j, \epsilon, D} - \int_{M_{s_i'}} u_{j, \epsilon, D}| <  \delta$ and that Ilmanen's analysis applies to the sequence $s_i'$. Because from the above any converging limit of $M_t$ must be contained in $\bigcup \Sigma_j$ if the sequence of $M_{s_i}$ converges as varifolds its support must be contained in $\bigcup_{j=1}^k\Sigma_j$, and so by ranging over $D, j$ and letting $\epsilon, \delta \to 0$ we thus get that if the sequence $M_{s_i}$ converges as varifolds the limit must be precisely $m_1\Sigma_1+\cdots+m_k\Sigma_k$ as claimed. \end{proof}

Indeed, if one of the minimal surfaces $\Sigma_i$ is merely stable but not strictly stable, the above contracting tubular neighborhood may not exist. An example of such a minimal surface can be found in \cite[Appendix B]{StevensSun24_LargeArea}. To wrap up this section, we summarize some topological properties of the limit minimal surfaces in the case $M_t$ can be approximated by generic flows: 
\begin{prop}\label{limit_topology}
    Suppose that $M_t$ is an inner or outermost flow of $M$ or is a generic flow whose long term limit is nonempty, and let $m_1\Sigma_1+\cdots+m_k\Sigma_k$ be limiting minimal surfaces as in Theorem \ref{smoothfate}, and that $N$ is oriented, implying 2--sidedness is equivalent to orientability. Then the following hold:
    \begin{enumerate}
     \item $m_1\Sigma_1+\cdots+m_k\Sigma_k$ in Theorem \ref{smoothfate} is homologous to $M$ in $H_2(N,\Z_2)$.
      \item If $\Sigma_j$ is $1$--sided, then $m_j$ is even. 
     \item When all the $\Sigma_j$ are orientable $\sum\limits_{j=1}^k m_j\cdot \text{genus}(\Sigma_j)\leq \text{genus}(M)$. More generally 
     \begin{equation}
     \sum\limits_{j, \text{ }\Sigma_j \text{ 2--sided}} m_i \cdot\text{genus}(\Sigma_j) + \sum\limits_{j,  \text{ }\Sigma_j \text{ 1--sided}} \frac{m_j}{2} \cdot (\text{genus}(\Sigma_j) - 1)\leq \text{genus}(M)
     \end{equation} 
   
     \end{enumerate}
\end{prop}

\begin{proof}
By Lemma \ref{lem:PreserveHomology}, item (1) holds because this homology is preserved under limits. Now for the other parts, we first note that because the convergence to the $\Sigma_j$ is multisheeted away from a discrete set of points, each sheet in the convergence corresponds to a connected component of the boundary of the tubular neighborhood of the limit (or, in other words, a connected component of its unit normal bundle which, for a surface $\Sigma$, in this case will have either one or two connected components depending on the orientability of $\Sigma$). If $\Sigma_j$ is $1$-sided, there is only one connected component, but about a point locally there are two components, giving item (2). For item (3), by genus of nonoriented $\Sigma_j$ we refer to the demigenus and, since this may be somewhat less familiar, first we consider the case where all the $\Sigma_j$ are orientable. In this case, because all of the limit surfaces $\Sigma_j$ are orientable, each sheet over $\Sigma_j$ in the convergence is homeomorphic to $\Sigma_j$ away from a finite set of points, following the discussion above; then the claim follows easily from Lemma \ref{lem:GenusNonincreasing}. In case some are nonorientable, we note by the classification of surfaces that they must be homeomorphic to a connected sum of projective planes, say $k_j$ of them. The demigenus of such a connect sum is $k_j$ and the double cover of such a connect sum, corresponding to a single sheet in the convergence, happens to have genus $k_j - 1$. Because these numbers agree, we obtain the formula in item (3).  \end{proof}

\section{The long--time limit is generically stable or empty: the proofs of Theorems \ref{Rc_thm} and \ref{pw_thm}}
From the previous section, we see that if the limit of an almost regular flow is nonempty, it converges to a collection of smooth minimal surfaces, possibly with multiplicity. We also obtain some rough topological information about the limit. Even so, it is still often desirable to specifically produce stable minimal surfaces because they carry more structure; for instance, as a basic but important example, we recall that stable surfaces in PSC $3$--manifolds must be diffeomorphic to $S^2$ or $\R P ^2$, but just in $S^3$ there are minimal surfaces of arbitrarily large genus. In this section, we show that with appropriately timed small isotopies, we can construct piecewise flows whose long-time limits are either empty or stable minimal surfaces.

We consider first the Ricci positive case. As promised in the introduction, arguing in this case turns out to be simple because, in this case, the avoidance principle says two disjoint flows actually "repel." In fact, we don't even need the regularity statements from the previous statement to proceed, and indeed, our argument works in all dimensions.

\begin{proof}[proof of Theorem \ref{Rc_thm}] Where $M$ and $N$ are as in the statement of the theorem, suppose that the flow $M_t$ of $M$ doesn't go extinct in finite time. Since the flow is gradient of the area functional, we may pick a sequence of times $t_i \to \infty$ for which $\int_{M_{t_i}} H^2 \to 0$.  After potentially passing to a subsequence, the sequence $M_{t_i}$ of surfaces limits to a varifold $\Sigma$ which, by the choice of $t_i$, must be stationary. Now consider a one sided perturbation $M'$ of $M$, for which say $d(M, M') > \delta$ for some $\delta > 0$. Since $N$ is compact and Ricci positive, there exists some $\Lambda > 0$ for which the Ricci curvature of $N$ is bounded below by $\Lambda$. Hence by the avoidance principle, Theorem \ref{ambient_avoidance} above, we have that the distance between $M'_t$ and $\Sigma$ is bounded below by $\delta e^{\Lambda t}$ for all times; because $N$ is compact so has bounded diameter for $t$ sufficiently large this gives that the level set flow $M'_t$ of $M'$ must be empty. 
\end{proof}

We point out that, as a consequence of the above, in an orientable $3$--manifold every $2$--cycle can be represented by a properly embedded $2$--sided surface (see, for instance, Lemma 3.6 in \cite{Hatcher07_Notes3Mfd}), and Proposition \ref{limit_topology} we have another proof of the following classical fact:
\begin{cor}\label{vanish} $H_{2}(N, \R)$ of a compact orientable $3$--manifold $N$ which admits a metric of positive Ricci curvature vanishes. 
\end{cor}

Indeed, by Hodge theory and the Bochner formula, $H^1(N)$ vanishes, so the claim follows by Poincar\'e duality. Now, in the case $N$ doesn't have Ricci positive curvature, the avoidance principle doesn't give as strong information, and indeed it is easy to concoct examples where two disjoint flows may converge to each other as $t \to \infty$. For instance, let $N$ be a compact manifold for which $H_2(N,\Z)$ is nontrivial (so cannot have positive Ricci curvature by the above). By the direct method in a nontrivial class of $H_2(N,\Z)$, there is an area minimizing surface $\Sigma$, and by perturbing the metric to be bumpy, we may even suppose it is strictly area minimizing. By considering disjoint nearby surfaces (sa,y graphical over $\Sigma$), their flows must converge to $\Sigma$ in the long term.

With this in mind, we now move on to Theorem \ref{pw_thm}. Recall the definition of piecewise generic flow, Definition \ref{pwflow_def}, and Theorem \ref{smoothfate}. If $M_t$ goes extinct or all the minimal surfaces $\Sigma_j$ it subconverges to, defined as in the previous section, are stable minimal surfaces, then we are done. Otherwise, we apply isotopes to $M_{t}$ at a smooth time far along the flow, which ensures the perturbation must not converge to the same surfaces after restarting the flow:

\begin{lem}\label{perturb} Let $M$ be a $2$--sided properly embedded surface in a closed compact Riemannian $3$--manifold $(N^3,g)$, and suppose $M_t$ is an almost regular flow out of $M$ which can be approximated by generic flows (if not generic itself, the inner or outermost flow from $M$). Suppose $M_{t_i}$ converges to $m_1\Sigma_1+\cdots+m_k\Sigma_k$ as varifolds, and one of $\Sigma_j$ is unstable. Then, for sufficiently large $t_i$, we may isotope $M_{t_i}$ so that any almost regular flow starting from the perturbed surface can not enter the tubular neighborhood of that unstable $\Sigma_j$.
\end{lem}

\begin{proof} Suppose $\Sigma_1$ is an unstable minimal surface, and in the following, we simply call it $\Sigma$. We first consider the situation that $\Sigma$ is unstable and $2$-sided. Then we can find a unit normal vector field $\nu$ over it, and we may assume $\varphi$ is the first eigenfunction of the linearized operator $L_\Sigma$. Let us use $(x,s)\in\Sigma\times(-\tau,\tau)$ to parametrize a tubular neighborhood $\cN_{\tau}$ of $\Sigma$ by $(x,s)\to\exp_x(s\varphi(x)\nu(x))$. We first define an isotope $\Psi_i$ of the ambient manifold as follows: we define the vector field $\pr_s$ in $\cN_{\tau/2}$ using the coordinate $(x,s)$, then we smoothly extend it to $0$ outside $\cN_{\tau}$ to get a smooth vector field $W$. Then we define $\{\Psi(w)\}_{w\in\R}$ to be the isotopy generated by the vector field $W$.

Let us choose a very small $\delta>0$, such that $\cN_\delta(\Sigma)$ is completely contained in the domain given by the parametrization $(x,s)\to\exp_x(s\varphi(x)\nu(x))$ for $s\in(-\tau/10,\tau/10)$, and $\Psi(\tau/2)\cN_\delta(\Sigma)$ will completely leave $\cN_\delta(\Sigma)$. Following Lemma \ref{lem:LimNbhd}, we may assume $M_{t_i}\subset \bigcup_{j=1}^k\cN_{\delta}(\Sigma_j)$ for sufficiently large $t_i$. Then $\tilde M_{t_i}:=\Psi(\tau/2)M_{t_i}$ must be completely disjoint from $\cN_\delta(\Sigma)$, and this is a $2$--sided flow from which we may safely restart the flow from. Recalling that the supports of Brakke flows are weak set flows. It's easy to see that any Brakke (and hence almost regular) flow starting from such $M' := \tilde M_{t_i}$ will never enter $\cN_{\delta}(\Sigma)$ again, by the avoidance principle. To see this, note that the domain generated by the parametrization $(x,s)\to\exp_x(s\varphi(x)\nu(x))$ for $s\in(-\tau/10,\tau/10)$ has compact, smooth and outwardly mean convex boundary, with two components, which we denote by $\Sigma_{\pm \tau/10}$. By the strict mean convexity and the compactness there exists $\delta, \eta > 0$ so that the smooth mean curvature flow of the boundary exists on $[0, \delta)$ and the distance between $\Sigma_{\pm \tau/10}$ and $(\Sigma'_{\pm \tau/10})_\delta$ is bounded below by $\eta$. The avoidance principle says $(\Sigma_{\pm \tau/10})_t$ and $M'_t$ are disjoint on $[0, \delta)$, and because $\eta > 0$ we can iterate this argument over the intervals $[\delta, 2\delta)$ et. cetera to get the claim.  

Finally, we briefly describe the situation that $\Sigma$ is unstable and $1$--sided. The argument is similar, but now the tubular neighborhood of $\Sigma$ is only parametrized by $\overline{\Sigma}\times(0,\tau)\bigcup \Sigma$, where $\overline{\Sigma}$ is the oriented double cover of $\Sigma$. In this case, we can still construct the isotopy to push the flow outside this tubular neighborhood, and verbatim arguments validate all the previous arguments.
\end{proof}

With this in hand, we are now ready to prove the main result of this section:

\begin{proof}[proof of Theorem \ref{pw_thm}] 
We may clearly restart the flow after the perturbation in Lemma \ref{perturb}, but a priori the set of times for which we perturb may accumulate to a finite time; we claim that only finitely many perturbations are required, though, to ensure that afterwards none of the $\Sigma_j$ could be unstable. Suppose $\Sigma$ is an unstable minimal surface that the flow was perturbed to avoid by the lemma; then if the flow $M_t$ starting from this perturbation converges to another unstable minimal surface it must converge to one which is disjoint from $\Sigma$ and hence disjoint from the component of the flow which was just perturbed -- but by the avoidance principle argument above the flow of this connected component may never converge to any such unstable minimal surface. 

Recall that the clearing out lemma \cite[12.2]{Ilmanen94_EllipReg} asserts that if the area of a flow $M_{t_0}$ inside $B_R(x)$ is sufficiently small, then after a short time, $M_{t_0+cR^2}$ will be disjoint from $B_{R/2}(x)$. Now, because we can take the parameter $\delta$ in the proof above as small as we wish, we may suppose the area of the piecewise mean curvature flow is bounded by, say, twice the area of $M$ after the $k$-th perturbation/iteration of Lemma \ref{perturb} for any $k \in \mathbb{N}$. In particular, by the clearing out lemma, only finitely many components of $M_t$ at a given smooth time might go on to flow to a minimal surface, and the rest of those with small area must quickly be exhausted. By this, we may crudely uniformly upper bound (independent of time) the number of such components. Alternatively, one may show there is a lower area bound on minimal surfaces of $N$ with bounded genus using \cite{White18_CptMinSurf3-Mfd} and use this. Applying the reasoning above to each such component then gives what we want.\end{proof}

\section{Applications to existence questions for minimal surfaces: proofs of corollaries \ref{existence1}, \ref{existence2}, \ref{existence3}, and \ref{existence5}}\label{applications}

In this section (aside from the last application, where we use Theorem \ref{longterm}), we consider various circumstances where we can show that the piecewise almost regular flows from the last section will be nonempty as $t \to \infty$ and say something about the topology of the limit. The first is on the existence of stable minimal spheres/projective planes when $\pi_2(N) \neq 0$.

\begin{proof}[proof of Corollary \ref{existence1}] First we note that by the sphere theorem in $3$--manifold topology \cite{Papakyriakopoulos57_Dehn} we can find, using $N$ is orientable, a homotopically nontrivial embedded $2$--sphere $M$. Of course, $M$ is itself orientable, and we may use it as initial data in Theorem \ref{pw_thm}. By considering the approximating surgery flows $M_t$ during smooth times will consist of a collection (possibly multiple, if there are nontrivial neckpinches) of embedded $2$--spheres. Now, either the flow from the conclusion will terminate in finite time or converge to a stable minimal surface $\Sigma$ which, by the nature of the convergence discussed in Theorem \ref{smoothfate} and the proof of Proposition \ref{limit_topology}, must be a collection of embedded stable minimal spheres or projective planes because each component for large (smooth) times is covered $k \geq 1$ times by $M_t $ for large $t$, away from mentioned bridge points and each component of this is an embedded $2$--sphere. Picking any of these components then gives the statement.

So, it suffices to rule out the flow going extinct in finite time. Supposing this is actually the case, considering the approximating surgery flow, one can show that $M$ is isotopic to a "marble" tree as in \cite{BuzanoHaslhoferHershkovits21_Moduli} or similarly \cite{Mramor18_Finiteness}, taking some care because in our setting the flow may not be mean convex even if the high curvature regions are. Because in the mean curvature flow with surgery for large parameters, the high curvature regions all bound handlebodies, one can show such marble trees bound $3$--balls in $N$, giving that $M$ is in fact homotopically trivial, leading to a contradiction.\end{proof}

We remind the reader that the next corollary is on the existence of stable minimal tori in the aspherical case. Corollary \ref{existence2} in fact follows from Corollary \ref{existence3}: it isn't hard to see by Dehn's lemma and the trivial $\pi_2$ assumption that if the torus $T$ isn't algebraically incompressible, then it bounds a $3$--cycle. That being said, we give the proof below for the sake of exposition, because it is somewhat simpler than the more general statement.

\begin{proof}[proof of Corollary \ref{existence2}] 
Denote by $M$ the torus in the statement. Because the torus is orientable and $N$ is assumed oriented, we have that $M$ is $2$--sided, so that we may apply Theorem \ref{pw_thm} to it. First, we argue that the flow must not go extinct. Supposing it does,  we note that approximating surgery flows go extinct as well. Considering one of these and denoting it by $S_t$, we note that high curvature regions which are discarded certainly bound $3$--cycles, homeomorphic to either $3$--balls or solid tori $\simeq S^1 \times D^2$, since different time slices along a smooth mean curvature flow are homologous this gives that $M$ is homologous to $S_t$ for all $t > 0$ and hence bounds a $3$--cycle, contradicting that $M$ is assumed to be homologically nontrivial. 

Next, we note that for any finite smooth time $t$, $M_t$ cannot be given by a union of embedded spheres $P_1, \ldots, P_k$. To see this, because $\pi_2(N)$ is trivial, each of the $P_i$ is nullhomotopic, and hence is the boundary of a $3$--cycle, so since $M_t$ is homologous to $M$, this gives that $M$ is homologically trivial, giving a contradiction. In particular, a connected component of $M_t$ at a smooth time must be homeomorphic to a torus. Now, because the flow doesn't go extinct as before we may produce a sequence of times $t_i \to \infty$ for which $M_{t_i}$ converge to an embedded stable minimal surface $\Sigma$, and we claim one of the connected components of $\Sigma$ is a torus or a Klein bottle (which is the nonorientable surface covered by the torus) -- note this doesn't follow immediately from the case that a component of $M_t$ is a torus. For example, one can take a standardly embedded torus in $\R^3$, tighten the "donut hole" and then push the sides of the inner cylinder along the outer sides of the donut so in this fashion covering the sphere twice, away from bridge points. Of course, if one of the components is a torus or Klein bottle, we are done, so it suffices to rule out the case that none of them are. By Proposition \ref{limit_topology}, in this case, all the components of $\Sigma$ are spheres or projective planes. 

Suppose first all the components are spheres: then by assumption they are nullhomotopic, and so each of these spheres could then be homotoped to be within coordinate patches homeomorphic to open balls. As these patches are homeomorphic to an open ball in $\R^3$, all surfaces contained within are nullhomologous. Using that homotopic surfaces are homologous as before, this implies $M$ itself is nullhomologous, giving a contradiction. If some of the components are projective planes, we can argue similarly, because the boundary of their tubular neighborhoods is diffeomorphic to spheres. 
\end{proof}

As the reader recalls from the introduction, the next statement concerns the existence of stable minimal surfaces in the presence of algebraically incompressible initial data, which is to say the inclusion map on fundamental groups injects; by Dehn's lemma, algebraically incompressible surfaces are incompressible in the standard sense, i.e., that there are no compressing discs in $N$ along $M$. This topological assumption is arguably most naturally suited to how the topology of a flow may change across singularities, at least generic ones. This is because the effect of flowing through a neckpinch is essentially to glue a thickened disc along the flow of the corresponding domain, and so if the surface is incompressible, only spheres can be pinched off.

\begin{proof}[proof of Corollary \ref{existence3}] 
Supposing that $M = \iota(\Sigma)$ is as in the statement, we consider, as before, the approximating surgery flow $S_t$ to $M_t$. We first claim that at any smooth time, a given connected component of it is either homeomorphic to either a sphere or $M$. Supposing not, and let $T$ be the smallest time, necessarily right after a surgery, that this isn't the case -- for a fixed choice of approximating surgery flow, this time is well defined. Then it's easy to see that the cross section curve $\gamma$ at one of the necks where surgery was performed on $S_T^{-}$ (the presurgery surface) is homotopically nontrivial in $M_T$. The surgery shows that $\gamma$ bounds a disc in $N$ though, so that $S_T^{-}$ is not algebraically incompressible. Since $T$ was the least such time, this implies that $M$ itself is not algebraically incompressible, which gives a contradiction.

Although it is the case that the flow preserves the topology of $M$ in this sense, it could be the case that there is strict genus drop in the limit as $t \to \infty$. In the following we focus only on the connected component $\tilde{M_t}$ of $M_t$ homeomorphic to $M$, and denote its corresponding limiting stable minimal surface $\Sigma$; the reason this suffices for us because, by the $\pi_2$ triviality assumption and that the other components pinched off the flow will be spheres, at any given time $M_t$ will be homotopic to just this component. Now, if the convergence is with multiplicity one, then it will be the case that $\Sigma$ is orientable and isotopic to $M$. Supposing this is not the case, then from the proof of Theorem \ref{pw_thm} we have that either $\Sigma$ is nonorientable and $M_t$ converges to $\Sigma$ with multiplicity $2$, or $M_t \to \Sigma$ with some multiplicity $m$ and there is at least one bridge point $p$, within which for some large time $s$ and some small $r > 0$ we have $M_s \cap B(p,r)$ has less than $m$ components -- our task is to rule out this latter case. Taking $r$ slightly larger if necessary $M \cap (B(p,2r) \setminus B(p, r))$ is comprised of $m$ disjoint annuli, where $m$ is the multiplicity of convergence; after a slight isotopy we may suppose these annuli $A_i$ lay in parallel planes and so may order them from "top" to "bottom" in the obvious sense, with $A_1$ being the top annulus. Without loss of generality, the top annulus at the point $p$ is nontrivial in the sense it doesn't border a disc in $M_s \cap B(p,r)$; if not below we can consider the next lowest disc and so on.

Considering one of the boundary components $\gamma$ of $A_1$, we may clearly find an embedded disc $D$ in $N$ disjoint from $M_ s$ with boundary $\gamma$. If $\gamma$ is nonseparating or splits $M_s$ into two components of genus greater than zero then $D$ will be a compressing disc because $\gamma$ will be homotopically nontrivial in $M$, which gives a contradiction. If the number of bridge points connected to the top sheet (i.e. where $B(p_i, r_i)$ are the bridge points and neighborhoods, the connected component of $M \setminus \cup_i B(p_i, r_i)$ containing $A_1$) is greater than one then $\gamma$ is nonseparating, so that we are done; let us suppose that $p$ is the only bridge point where the top sheet connects to the lower sheets then. If $\Sigma$ is not an embedded sphere or projective plane, then $\gamma$ would disconnect $M$ into two components of nonzero genus, completing the argument in this case. If $\Sigma$ is a sphere or projective plane, a bit more arguing is required. As before, first suppose that $\Sigma$ is a sphere. Then, because $N$ has trivial $\pi_2$, we may homotope $\Sigma$ and hence $M_s$ to lie in a small coordinate chart; of course, because $\R^3$ is contractible, $M_s$ in this case cannot be incompressible. If it is a projective plane again, the boundary of its tubular neighborhood is a sphere, and we may proceed as in the first case. 
\end{proof}

Leaving the world of stable minimal surfaces, we now consider the use of our first convergence statement, Theorem \ref{longterm}, to essentially replace the tightening process in minmax. Recalling that in the setting of Corollary \ref{existence5} setting $N \simeq S^3$ we may consider the standard one parameter sweepout $\Sigma_s$ of it by $2$--spheres, where there are two singular values of the sweepout that are single points corresponding to the north and south poles under a fixed choice of diffeomorphism from $N$ to $S^3$. Perhaps the most natural thing to do then is to consider the flow of the sweepout under the flow, by which we mean the family of the flows $(\Sigma_s)_t$ of the individual leaves $\Sigma_s$. Recalling that the width of any sweepout of $S^3$ by spheres is positive by the isoperimetric inequality, we see if we could show that if the flow will remain a sweepout, then Theorem \ref{longterm} says that the flow of at least one leaf will not go extinct, and because $S^3$ is simply connected and embedded hypersurfaces in such spaces are orientable Proposition \ref{limit_topology} gives the limit will be an embedded minimal sphere, as usual potentially with multiplicity. There are some issues with the approach though, unfortunately. 

The biggest issue in general appears to be that the level set flow of some of the leaves may fatten, but in our particular setting, this may be ruled out as we discuss below. There is also an issue with preserving the typical regularity assumptions one makes in the definition of sweepout. Recall that a sweepout of $N$ (here, $1$--parameter) is a family of surfaces $\Sigma_s$, $s \in [0,1]$, which for finitely many $s$ is singular in a finite set of points and smooth elsewhere and whose union is all of $N$. Now, we know that for a fixed leaf $\Sigma_s$ the flow $(\Sigma_s)_t$ will be smooth for almost all times $t$, but because we wish to consider the flow of the family, parameterized over the whole unit interval (which is uncountable so Baire category theorem cannot apply), it doesn't seem clear if we can show the flow of the whole family will be smooth for almost all or even many times. The current state of the art also doesn't rule out the size of the singular set of a flow being greater than a finite union of points, either.

\begin{proof}[proof of Corollary \ref{existence5}] 
With respect to the discussion above, one may be able to proceed by generalizing the results of minmax theory to more general sweepouts, including those "generated" by the flow, but instead, we will take a somewhat more hands-on approach. Denote by $E(s)$ the extinction time of the flow of $\Sigma_s$; if this never occurs, of course, we set $E(s) = \infty$ so it is a function into the extended reals.  We next claim that $E(s)$ is sequentially continuous. To see this, we first note that because each of the $\Sigma_s$ is a sphere, their singularities under the flow must be mean convex and so each of their flows is nonfattening -- implying, for instance, that the almost regular flows $(\Sigma_{s_i})_t $ are well defined. Fixing an $s_0 \in [0,1]$, nonfattening also implies that for a sequence $s_i \to s_0$, after potentially passing to a subsequence by Brakke compactness, $\lim\limits_{s_j \to s_0} (\Sigma_{s_i})_t \to (\Sigma_{s_0})_t$ as Brakke flows because $\lim\limits_{s_j \to s_0} (\Sigma_{s_i})_t $ is a weak flow associated to $\Sigma_{s_0}$. With this in mind, suppose then that there was a sequence $s_i \to s_0$ and $\epsilon > 0$ for which $E(s_0) < A - \epsilon$, where $A = \lim\limits_{i \to \infty} E(s_i)$. 

At time $T = E(s_0)$, by Colding and Minicozzi \cite{ColdingMinicozzi16_SingularSet} $(\Sigma_{s_0})_T$ is the union of a finite number of compact embedded Lipschitz curves corresponding to cylindrical singularities along with a countable collection of points corresponding to spherical singularities. This gives that for $i$ sufficiently large we make take the mass of $(\Sigma_{s_i})_T$ to be as small as we like, so in particular the clearing out lemma says these flows must go extinct before $t = A - \epsilon/2$, giving a contradiction to show that $E(s)$ is upper semicontinuous. Similarly, $E(s)$ is lower semicontinuous, giving our claim of sequential continuity.

Because $E(s)$ is sequentially continuous, by the compactness of the unit interval, $\sup\limits_{s \in [0,1]} E(s) = T$ is obtained for some time $s_0$, recycling notation. Now, as above, we have $(\Sigma_{s_0})_T$ must consist of the union of a finite number of compact embedded Lipschitz curves corresponding to cylindrical singularities, along with a countable collection of points corresponding to spherical singularities. Picking some smooth time $t^*$ right before $T$, the resolution of the mean convex neighborhood conjecture gives that $(\Sigma_{s_0})_{t_*}$ is mean convex and after rescaling is locally modeled on bowl solitons, round cylinders, or round spheres. 

To see why this is useful, first note that $s_0 \in (0,1)$: this is because for $s$ very near $0$ or $1$ that $\Sigma_s$ are approximately small concentric round spheres contained in a coordinate chart so that $\lim\limits_{s \to 0, 1} E(s) = 0$; on the other hand for a fixed smooth $\Sigma_s$ its smooth flow exists for a short but strictly positive time. Secondly, the flows $(\Sigma_{s})_t$ are pairwise disjoint in $s$ because $\Sigma_s$ are and almost regular flows are weak set flows. In particular, for $s - s_0$ sufficiently small $s$ Brakke regularity says that $(\Sigma_s)_{t^*}$ is a graph over $(\Sigma_{s_0})_{t_*}$ laying on one side or the other of it, and we can find $s$ so that the flow at time $t^*$ is on either side. Hence, we see that for $s$ either slightly greater or smaller than $s_0$, $(\Sigma_{s})_T$ must not be singular, or in other words, that their flow lasts at least slightly after time $T$, giving a contradiction to the definition of $T$. 
\end{proof}
\bibliographystyle{alpha}
\bibliography{GMT}

@article{AngenentIlmanenVelazquez17_fattening,
    author = {Angenent, S. B. and T. Ilmanen and Vel\'{a}zquez, J. J. L.},
    title = {Fattening from smooth initial data in mean curvature flow},
    journal = {Preprint},
    year = {2017},
    URL = {https://people.math.wisc.edu/~angenent/preprints/unfinished/fattening.pdf},
}

@book{Brakke78,
	url = {https://doi.org/10.1515/9781400867431},
	title = {The Motion of a Surface by Its Mean Curvature. (MN-20)},
	author = {Kenneth A. Brakke},
	publisher = {Princeton University Press},
	address = {Princeton},
	doi = {doi:10.1515/9781400867431},
	isbn = {9781400867431},
	year = {1978},
}

@article {Brendle16_genus0,
    AUTHOR = {Brendle, Simon},
     TITLE = {Embedded self-similar shrinkers of genus 0},
   JOURNAL = {Ann. of Math. (2)},
  FJOURNAL = {Annals of Mathematics. Second Series},
    VOLUME = {183},
      YEAR = {2016},
    NUMBER = {2},
     PAGES = {715--728},
      ISSN = {0003-486X,1939-8980},
   MRCLASS = {53C25 (53C44)},
  MRNUMBER = {3450486},
MRREVIEWER = {Casey\ Lynn\ Kelleher},
       DOI = {10.4007/annals.2016.183.2.6},
       URL = {https://doi.org/10.4007/annals.2016.183.2.6},
}

@article{ChenGigaGoto91_LSF,
	title={Uniqueness and existence of viscosity solutions of generalized mean curvature flow equations},
	author={Chen, Yun Gang and Giga, Yoshikazu and Goto, Shun'ichi},
	journal={Journal of differential geometry},
	volume={33},
	number={3},
	pages={749--786},
	year={1991},
	publisher={Lehigh University}
}

@article {CCMS20_GenericMCF,
    AUTHOR = {Chodosh, Otis and Choi, Kyeongsu and Mantoulidis, Christos and
              Schulze, Felix},
     TITLE = {Mean curvature flow with generic initial data},
   JOURNAL = {Invent. Math.},
  FJOURNAL = {Inventiones Mathematicae},
    VOLUME = {237},
      YEAR = {2024},
    NUMBER = {1},
     PAGES = {121--220},
      ISSN = {0020-9910,1432-1297},
   MRCLASS = {53E10 (35K93)},
  MRNUMBER = {4756990},
       DOI = {10.1007/s00222-024-01258-0},
       URL = {https://doi.org/10.1007/s00222-024-01258-0},
}

@article{chodoshchoischulze2023mean,
  title={Mean curvature flow with generic initial data II},
  author={Chodosh, Otis and Choi, Kyeongsu and Schulze, Felix},
  journal={arXiv preprint arXiv:2302.08409},
  year={2023}
}

@article {ChoiHaslhoferHershkovits18_MeanConvNeighb,
    AUTHOR = {Choi, Kyeongsu and Haslhofer, Robert and Hershkovits, Or},
     TITLE = {Ancient low-entropy flows, mean-convex neighborhoods, and
              uniqueness},
   JOURNAL = {Acta Math.},
  FJOURNAL = {Acta Mathematica},
    VOLUME = {228},
      YEAR = {2022},
    NUMBER = {2},
     PAGES = {217--301},
      ISSN = {0001-5962,1871-2509},
   MRCLASS = {53E10},
  MRNUMBER = {4448681},
MRREVIEWER = {Shu-Yu\ Hsu},
       DOI = {10.4310/acta.2022.v228.n2.a1},
       URL = {https://doi.org/10.4310/acta.2022.v228.n2.a1},
}

@article{ColdingMinicozzi12_generic,
	title={Generic mean curvature flow I; generic singularities},
	author={Colding, Tobias H and Minicozzi, William P},
	journal={Annals of mathematics},
	pages={755--833},
	year={2012},
	publisher={JSTOR}
}

@article{ColdingMinicozzi16_SingularSet,
  title={The singular set of mean curvature flow with generic singularities},
  author={Colding, Tobias Holck and Minicozzi, William P},
  journal={Inventiones mathematicae},
  volume={204},
  number={2},
  pages={443--471},
  year={2016},
  publisher={Springer}
}

@article{EckerHuisken89_EntireGraph,
	title={Mean curvature evolution of entire graphs},
	author={Ecker, Klaus and Huisken, Gerhard},
	journal={Annals of Mathematics},
	volume={130},
	number={3},
	pages={453--471},
	year={1989},
	publisher={JSTOR}
}

@article{EvansSpruck91,
	title={Motion of level sets by mean curvature. I},
	author={Evans, Lawrence C and Spruck, Joel},
	journal={Journal of Differential Geometry},
	volume={33},
	number={3},
	pages={635--681},
	year={1991},
	publisher={Lehigh University}
}

@article{HershkovitsWhite20_Nonfattening,
	title={Nonfattening of mean curvature flow at singularities of mean convex type},
	author={Hershkovits, Or and White, Brian},
	journal={Communications on Pure and Applied Mathematics},
	volume={73},
	number={3},
	pages={558--580},
	year={2020},
	publisher={Wiley Online Library}
}

@article{Ilmanen92_LSF,
	title={Generalized flow of sets by mean curvature on a manifold},
	author={Ilmanen, Tom},
	journal={Indiana University mathematics journal},
	pages={671--705},
	year={1992},
	publisher={JSTOR}
}

@book{Ilmanen94_EllipReg,
	title={Elliptic regularization and partial regularity for motion by mean curvature},
	author={Ilmanen, Tom},
	volume={520},
	year={1994},
	publisher={American Mathematical Soc.}
}

@article{Ilmanen95_Sing2D,
	title={Singularities of mean curvature flow of surfaces},
	author={Ilmanen, Tom},
	journal={preprint},
	year={1995}
}

@article{Ketover24_Fattening,
  title={Self-shrinkers whose asymptotic cones fatten},
  author={Ketover, Daniel},
  journal={arXiv preprint arXiv:2407.01240},
  year={2024}
}

@article{LeeZhao24_MCFconical,
  title={Closed mean curvature flows with asymptotically conical singularities},
  author={Lee, Tang-Kai and Zhao, Xinrui},
  journal={arXiv preprint arXiv:2405.15577},
  year={2024}
}

@Article{Simon83Asym,
	title = {Asymptotics for a {{Class}} of {{Non}}-{{Linear Evolution Equations}}, with {{Applications}} to {{Geometric Problems}}},
	author = {Simon, Leon},
	year = {1983},
	volume = {118},
	pages = {525--571},
	issn = {0003-486X},
	doi = {10.2307/2006981},
	journal = {Annals of Mathematics},
	number = {3}
}

@article {sun-generic-multi-1,
    AUTHOR = {Sun, Ao},
     TITLE = {Local entropy and generic multiplicity one singularities of
              mean curvature flow of surfaces},
   JOURNAL = {J. Differential Geom.},
  FJOURNAL = {Journal of Differential Geometry},
    VOLUME = {124},
      YEAR = {2023},
    NUMBER = {1},
     PAGES = {169--198},
      ISSN = {0022-040X,1945-743X},
   MRCLASS = {53E10},
  MRNUMBER = {4593902},
       DOI = {10.4310/jdg/1685121322},
       URL = {https://doi.org/10.4310/jdg/1685121322},
}

@article{SunXue2021_initial_closed,
	title={Initial Perturbation of the Mean Curvature Flow for closed limit shrinker},
	author={Sun, Ao and Xue, Jinxin},
	journal={arXiv preprint arXiv:2104.03101},
	year={2021}
}

@article{SunXue2021_initial_conical,
	title={Initial Perturbation of the Mean Curvature Flow for Asymptotical Conical Limit Shrinker},
	author={Sun, Ao and Xue, Jinxin},
	journal={arXiv preprint arXiv:2107.05066},
	year={2021}
}

@ARTICLE{White91_Bumpy,
	author = {White, Brian},
	title = {The space of minimal submanifolds for varying {R}iemannian metrics},
	journal = {Indiana Univ. Math. J.},
	fjournal = {Indiana University Mathematics Journal},
	volume = {40},
	year = {1991},
	number = {1},
	pages = {161--200},
	issn = {0022-2518},
	mrclass = {58D10 (53C42)},
	mrnumber = {1101226},
	mrreviewer = {João Lucas Marques Barbosa},
	doi = {10.1512/iumj.1991.40.40008},
	zblnumber = {0742.58009},
}

@article{White95_WSF_Top,
	title={The topology of hypersurfaces moving by mean curvature},
	author={White, Brian},
	journal={Communications in analysis and geometry},
	volume={3},
	number={2},
	pages={317--333},
	year={1995},
	publisher={International Press of Boston}
}

@article{White97_Stratif,
	author = {White, Brian},
	journal = {Journal für die reine und angewandte Mathematik},
	pages = {1-36},
	title = {Stratification of minimal surfaces, mean curvature flows, and harmonic maps.},
	url = {http://eudml.org/doc/153922},
	volume = {488},
	year = {1997},
}

@inproceedings {White02ICM,
	AUTHOR = {White, Brian},
	TITLE = {Evolution of curves and surfaces by mean curvature},
	BOOKTITLE = {Proceedings of the {I}nternational {C}ongress of
		{M}athematicians, {V}ol. {I} ({B}eijing, 2002)},
	PAGES = {525--538},
	PUBLISHER = {Higher Ed. Press, Beijing},
	YEAR = {2002},
	MRCLASS = {53C44},
	MRNUMBER = {1989203},
	MRREVIEWER = {Tommaso Pacini},
}

@article{White09_CurrentsVarifolds,
	title={Currents and flat chains associated to varifolds, with an application to mean curvature flow},
	author={White, Brian},
	journal={Duke Mathematical Journal},
	volume={148},
	number={1},
	pages={41--62},
	year={2009},
	publisher={Duke University Press}
}

@ARTICLE{White17_Bumpy,
	author = {White, Brian},
	title = {On the bumpy metrics theorem for minimal submanifolds},
	journal = {Amer. J. Math.},
	fjournal = {American Journal of Mathematics},
	volume = {139},
	year = {2017},
	number = {4},
	pages = {1149--1155},
	issn = {0002-9327},
	mrclass = {53C42},
	mrnumber = {3689325},
	doi = {10.1353/ajm.2017.0029},
}

@article{Zhou19_multi1,
	AUTHOR = {Zhou, Xin},
	TITLE = {On the multiplicity one conjecture in min-max theory},
	JOURNAL = {Ann. of Math. (2)},
	FJOURNAL = {Annals of Mathematics. Second Series},
	VOLUME = {192},
	YEAR = {2020},
	NUMBER = {3},
	PAGES = {767--820},
	ISSN = {0003-486X},
	MRCLASS = {53C42 (49J35 49Q05 58E12)},
	MRNUMBER = {4172621},
	DOI = {10.4007/annals.2020.192.3.3},
	URL = {https://doi.org/10.4007/annals.2020.192.3.3},
	Zbl = {1461.53051},
}

@article {HaslhoferKleiner17_MCFsurgery,
    AUTHOR = {Haslhofer, Robert and Kleiner, Bruce},
     TITLE = {Mean curvature flow with surgery},
   JOURNAL = {Duke Math. J.},
  FJOURNAL = {Duke Mathematical Journal},
    VOLUME = {166},
      YEAR = {2017},
    NUMBER = {9},
     PAGES = {1591--1626},
      ISSN = {0012-7094,1547-7398},
   MRCLASS = {53C44 (35K93 53C21)},
  MRNUMBER = {3662439},
MRREVIEWER = {Valentina-Mira\ Wheeler},
       DOI = {10.1215/00127094-0000008X},
       URL = {https://doi.org/10.1215/00127094-0000008X},
}

@article {BrendleHuisken16_MCFSurgeryR3,
    AUTHOR = {Brendle, Simon and Huisken, Gerhard},
     TITLE = {Mean curvature flow with surgery of mean convex surfaces in
              {$\Bbb R^3$}},
   JOURNAL = {Invent. Math.},
  FJOURNAL = {Inventiones Mathematicae},
    VOLUME = {203},
      YEAR = {2016},
    NUMBER = {2},
     PAGES = {615--654},
      ISSN = {0020-9910,1432-1297},
   MRCLASS = {53C44 (53A10)},
  MRNUMBER = {3455158},
MRREVIEWER = {James\ Alexander\ McCoy},
       DOI = {10.1007/s00222-015-0599-3},
       URL = {https://doi.org/10.1007/s00222-015-0599-3},
}

@article {Daniels-Holgate22_SurgeryApprox,
    AUTHOR = {Daniels-Holgate, J. M.},
     TITLE = {Approximation of mean curvature flow with generic
              singularities by smooth flows with surgery},
   JOURNAL = {Adv. Math.},
  FJOURNAL = {Advances in Mathematics},
    VOLUME = {410},
      YEAR = {2022},
     PAGES = {Paper No. 108715, 42},
      ISSN = {0001-8708,1090-2082},
   MRCLASS = {35K93 (53E10)},
  MRNUMBER = {4493003},
       DOI = {10.1016/j.aim.2022.108715},
       URL = {https://doi.org/10.1016/j.aim.2022.108715},
}

@article {Grayson89_CSF,
    AUTHOR = {Grayson, Matthew A.},
     TITLE = {Shortening embedded curves},
   JOURNAL = {Ann. of Math. (2)},
  FJOURNAL = {Annals of Mathematics. Second Series},
    VOLUME = {129},
      YEAR = {1989},
    NUMBER = {1},
     PAGES = {71--111},
      ISSN = {0003-486X,1939-8980},
   MRCLASS = {53C22 (58E10)},
  MRNUMBER = {979601},
MRREVIEWER = {Gudlaugur\ Thorbergsson},
       DOI = {10.2307/1971486},
       URL = {https://doi.org/10.2307/1971486},
}

@article{BamlerKleiner23_Multiplicity1,
  title={On the multiplicity one conjecture for mean curvature flows of surfaces},
  author={Bamler, Richard H and Kleiner, Bruce},
  journal={arXiv preprint arXiv:2312.02106},
  year={2023}
}

@article{ChenSun24_Multi2inManifold,
  title={Mean curvature flow with multiplicity $2 $ convergence in closed manifolds},
  author={Chen, Jingwen and Sun, Ao},
  journal={arXiv preprint arXiv:2402.04521},
  year={2024}
}

@article {SchoenYau79_ExistenceIncompressibleMinSurf,
    AUTHOR = {Schoen, R. and Yau, Shing Tung},
     TITLE = {Existence of incompressible minimal surfaces and the topology
              of three-dimensional manifolds with nonnegative scalar
              curvature},
   JOURNAL = {Ann. of Math. (2)},
  FJOURNAL = {Annals of Mathematics. Second Series},
    VOLUME = {110},
      YEAR = {1979},
    NUMBER = {1},
     PAGES = {127--142},
      ISSN = {0003-486X},
   MRCLASS = {58E12 (49F10 53C42)},
  MRNUMBER = {541332},
MRREVIEWER = {Jonathan\ Sacks},
       DOI = {10.2307/1971247},
       URL = {https://doi.org/10.2307/1971247},
}

@article {MeeksSimonYau82_MinSurf,
    AUTHOR = {Meeks, III, William and Simon, Leon and Yau, Shing Tung},
     TITLE = {Embedded minimal surfaces, exotic spheres, and manifolds with
              positive {R}icci curvature},
   JOURNAL = {Ann. of Math. (2)},
  FJOURNAL = {Annals of Mathematics. Second Series},
    VOLUME = {116},
      YEAR = {1982},
    NUMBER = {3},
     PAGES = {621--659},
      ISSN = {0003-486X},
   MRCLASS = {53C42 (49F10 53A10)},
  MRNUMBER = {678484},
MRREVIEWER = {Jo\~ao\ Lucas Marques Barbosa},
       DOI = {10.2307/2007026},
       URL = {https://doi.org/10.2307/2007026},
}

@article {FreedmanHassScott83_LeastArea,
    AUTHOR = {Freedman, Michael and Hass, Joel and Scott, Peter},
     TITLE = {Least area incompressible surfaces in {$3$}-manifolds},
   JOURNAL = {Invent. Math.},
  FJOURNAL = {Inventiones Mathematicae},
    VOLUME = {71},
      YEAR = {1983},
    NUMBER = {3},
     PAGES = {609--642},
      ISSN = {0020-9910,1432-1297},
   MRCLASS = {57N10 (53A10)},
  MRNUMBER = {695910},
MRREVIEWER = {John\ Hempel},
       DOI = {10.1007/BF02095997},
       URL = {https://doi.org/10.1007/BF02095997},
}

@article {SacksUhlenbeck81_minimal2sphere,
    AUTHOR = {Sacks, J. and Uhlenbeck, K.},
     TITLE = {The existence of minimal immersions of {$2$}-spheres},
   JOURNAL = {Ann. of Math. (2)},
  FJOURNAL = {Annals of Mathematics. Second Series},
    VOLUME = {113},
      YEAR = {1981},
    NUMBER = {1},
     PAGES = {1--24},
      ISSN = {0003-486X},
   MRCLASS = {58E12 (53C42 58E20)},
  MRNUMBER = {604040},
MRREVIEWER = {John\ C.\ Wood},
       DOI = {10.2307/1971131},
       URL = {https://doi.org/10.2307/1971131},
}

@article {FedererFleming60_Current,
    AUTHOR = {Federer, Herbert and Fleming, Wendell H.},
     TITLE = {Normal and integral currents},
   JOURNAL = {Ann. of Math. (2)},
  FJOURNAL = {Annals of Mathematics. Second Series},
    VOLUME = {72},
      YEAR = {1960},
     PAGES = {458--520},
      ISSN = {0003-486X},
   MRCLASS = {28.80 (53.45)},
  MRNUMBER = {123260},
MRREVIEWER = {L.\ C.\ Young},
       DOI = {10.2307/1970227},
       URL = {https://doi.org/10.2307/1970227},
}

@article {White24_Avoidance,
    AUTHOR = {White, Brian},
     TITLE = {The avoidance principle for noncompact hypersurfaces moving by
              mean curvature flow},
   JOURNAL = {Calc. Var. Partial Differential Equations},
  FJOURNAL = {Calculus of Variations and Partial Differential Equations},
    VOLUME = {63},
      YEAR = {2024},
    NUMBER = {5},
     PAGES = {Paper No. 111, 20},
      ISSN = {0944-2669,1432-0835},
   MRCLASS = {53E10},
  MRNUMBER = {4737267},
       DOI = {10.1007/s00526-024-02725-5},
       URL = {https://doi.org/10.1007/s00526-024-02725-5},
}

@article{IlmanenWhite24_Fattening,
  title={Fattening in mean curvature flow},
  author={Ilmanen, Tom and White, Brian},
  journal={arXiv preprint arXiv:2406.18703},
  year={2024}
}

@article {Lauer13_MCFSurgery,
    AUTHOR = {Lauer, Joseph},
     TITLE = {Convergence of mean curvature flows with surgery},
   JOURNAL = {Comm. Anal. Geom.},
  FJOURNAL = {Communications in Analysis and Geometry},
    VOLUME = {21},
      YEAR = {2013},
    NUMBER = {2},
     PAGES = {355--363},
      ISSN = {1019-8385,1944-9992},
   MRCLASS = {53C44},
  MRNUMBER = {3043750},
MRREVIEWER = {Esther\ Cabezas Rivas},
       DOI = {10.4310/CAG.2013.v21.n2.a4},
       URL = {https://doi.org/10.4310/CAG.2013.v21.n2.a4},
}

@article {Head13_MCF2convex,
    AUTHOR = {Head, John},
     TITLE = {On the mean curvature evolution of two-convex hypersurfaces},
   JOURNAL = {J. Differential Geom.},
  FJOURNAL = {Journal of Differential Geometry},
    VOLUME = {94},
      YEAR = {2013},
    NUMBER = {2},
     PAGES = {241--266},
      ISSN = {0022-040X,1945-743X},
   MRCLASS = {53C44 (35K55 35K93 53C21)},
  MRNUMBER = {3080482},
MRREVIEWER = {Kin\ Ming\ Hui},
       URL = {http://projecteuclid.org/euclid.jdg/1367438649},
}

@article {HaslhoferKetover19_Min2Sphere,
    AUTHOR = {Haslhofer, Robert and Ketover, Daniel},
     TITLE = {Minimal 2-spheres in 3-spheres},
   JOURNAL = {Duke Math. J.},
  FJOURNAL = {Duke Mathematical Journal},
    VOLUME = {168},
      YEAR = {2019},
    NUMBER = {10},
     PAGES = {1929--1975},
      ISSN = {0012-7094,1547-7398},
   MRCLASS = {49Q05 (49J35 53E10 58E12)},
  MRNUMBER = {3983295},
MRREVIEWER = {Hung\ Thanh\ Tran},
       DOI = {10.1215/00127094-2019-0009},
       URL = {https://doi.org/10.1215/00127094-2019-0009},
}

@article {Simon94_Willmore,
    AUTHOR = {Simon, Leon},
     TITLE = {Existence of surfaces minimizing the {W}illmore functional},
   JOURNAL = {Comm. Anal. Geom.},
  FJOURNAL = {Communications in Analysis and Geometry},
    VOLUME = {1},
      YEAR = {1993},
    NUMBER = {2},
     PAGES = {281--326},
      ISSN = {1019-8385,1944-9992},
   MRCLASS = {58E12 (49Q10 53A10)},
  MRNUMBER = {1243525},
MRREVIEWER = {J.\ E.\ Brothers},
       DOI = {10.4310/CAG.1993.v1.n2.a4},
       URL = {https://doi.org/10.4310/CAG.1993.v1.n2.a4},
}

@article {StevensSun24_LargeArea,
    AUTHOR = {Stevens, James and Sun, Ao},
     TITLE = {Existence of minimal hypersurfaces with arbitrarily large area
              and possible obstructions},
   JOURNAL = {J. Funct. Anal.},
  FJOURNAL = {Journal of Functional Analysis},
    VOLUME = {287},
      YEAR = {2024},
    NUMBER = {6},
     PAGES = {Paper No. 110526, 40},
      ISSN = {0022-1236,1096-0783},
   MRCLASS = {53A10},
  MRNUMBER = {4755022},
       DOI = {10.1016/j.jfa.2024.110526},
       URL = {https://doi.org/10.1016/j.jfa.2024.110526},
}

@article {White18_CptMinSurf3-Mfd,
    AUTHOR = {White, Brian},
     TITLE = {On the compactness theorem for embedded minimal surfaces in
              3-manifolds with locally bounded area and genus},
   JOURNAL = {Comm. Anal. Geom.},
  FJOURNAL = {Communications in Analysis and Geometry},
    VOLUME = {26},
      YEAR = {2018},
    NUMBER = {3},
     PAGES = {659--678},
      ISSN = {1019-8385,1944-9992},
   MRCLASS = {53A10 (49Q05)},
  MRNUMBER = {3844118},
MRREVIEWER = {Hojoo\ Lee},
       DOI = {10.4310/CAG.2018.v26.n3.a7},
       URL = {https://doi.org/10.4310/CAG.2018.v26.n3.a7},
}

@article {Papakyriakopoulos57_Dehn,
    AUTHOR = {Papakyriakopoulos, C. D.},
     TITLE = {On {D}ehn's lemma and the asphericity of knots},
   JOURNAL = {Ann. of Math. (2)},
  FJOURNAL = {Annals of Mathematics. Second Series},
    VOLUME = {66},
      YEAR = {1957},
     PAGES = {1--26},
      ISSN = {0003-486X},
   MRCLASS = {55.0X},
  MRNUMBER = {90053},
MRREVIEWER = {R.\ H.\ Fox},
       DOI = {10.2307/1970113},
       URL = {https://doi.org/10.2307/1970113},
}

@article {Mramor18_Finiteness,
    AUTHOR = {Mramor, Alexander},
     TITLE = {A finiteness theorem via the mean curvature flow with surgery},
   JOURNAL = {J. Geom. Anal.},
  FJOURNAL = {Journal of Geometric Analysis},
    VOLUME = {28},
      YEAR = {2018},
    NUMBER = {4},
     PAGES = {3348--3372},
      ISSN = {1050-6926,1559-002X},
   MRCLASS = {53C44},
  MRNUMBER = {3881975},
MRREVIEWER = {Mijia\ Lai},
       DOI = {10.1007/s12220-017-9962-5},
       URL = {https://doi.org/10.1007/s12220-017-9962-5},
}

@article {AM_CMgenericambient,
    AUTHOR = {Mramor, Alexander},
     TITLE = {Entropy and generic mean curvature flow in curved ambient spaces},
   JOURNAL = {Proc. Amer. Math. Soc.},
  FJOURNAL = {Proceedings of American Mathematical Society},
    VOLUME = {146},
      YEAR = {2018},
    NUMBER = {6},
     PAGES = {2663-2677},
      ISSN = {1050-6926,1559-002X},
   MRCLASS = {53C44},
  MRNUMBER = {MR3778166},
MRREVIEWER = {Mijia\ Lai},
       DOI = {10.1007/s12220-017-9962-5},
       URL = {https://doi.org/10.1090/proc/13964},
}

@article {BuzanoHaslhoferHershkovits21_Moduli,
    AUTHOR = {Buzano, Reto and Haslhofer, Robert and Hershkovits, Or},
     TITLE = {The moduli space of two-convex embedded spheres},
   JOURNAL = {J. Differential Geom.},
  FJOURNAL = {Journal of Differential Geometry},
    VOLUME = {118},
      YEAR = {2021},
    NUMBER = {2},
     PAGES = {189--221},
      ISSN = {0022-040X,1945-743X},
   MRCLASS = {58D27 (53E10)},
  MRNUMBER = {4278693},
MRREVIEWER = {Yong\ Wei},
       DOI = {10.4310/jdg/1622743139},
       URL = {https://doi.org/10.4310/jdg/1622743139},
}

@misc{Hatcher07_Notes3Mfd,
  title={Notes on basic 3-manifold topology},
  author={Hatcher, Allen},
  year={2007},
  url={https://pi.math.cornell.edu/~hatcher/3M/3M.pdf},
}

@incollection {ColdingDeLellis03_minmax,
    AUTHOR = {Colding, Tobias H. and De Lellis, Camillo},
     TITLE = {The min-max construction of minimal surfaces},
 BOOKTITLE = {Surveys in differential geometry, {V}ol.\ {VIII} ({B}oston,
              {MA}, 2002)},
    SERIES = {Surv. Differ. Geom.},
    VOLUME = {8},
     PAGES = {75--107},
 PUBLISHER = {Int. Press, Somerville, MA},
      YEAR = {2003},
      ISBN = {1-57146-114-0},
   MRCLASS = {53A10 (49Q05 53C42)},
  MRNUMBER = {2039986},
MRREVIEWER = {Fei-Tsen\ Liang},
       DOI = {10.4310/SDG.2003.v8.n1.a3},
       URL = {https://doi.org/10.4310/SDG.2003.v8.n1.a3},
}

@article{WangZhou23_4Sphere,
  title={Existence of four minimal spheres in $ {S}^3$ with a bumpy metric},
  author={Wang, Zhichao and Zhou, Xin},
  journal={arXiv preprint arXiv:2305.08755},
  year={2023}
}

@article {MeeksYau80_Top3d,
    AUTHOR = {Meeks, III, William H. and Yau, Shing Tung},
     TITLE = {Topology of three-dimensional manifolds and the embedding
              problems in minimal surface theory},
   JOURNAL = {Ann. of Math. (2)},
  FJOURNAL = {Annals of Mathematics. Second Series},
    VOLUME = {112},
      YEAR = {1980},
    NUMBER = {3},
     PAGES = {441--484},
      ISSN = {0003-486X},
   MRCLASS = {53C42 (49F10 57M35)},
  MRNUMBER = {595203},
MRREVIEWER = {F.\ J.\ Almgren, Jr.},
       DOI = {10.2307/1971088},
       URL = {https://doi.org/10.2307/1971088},
}

@article {DunfieldTHurston06_Random3Mfd,
    AUTHOR = {Dunfield, Nathan M. and Thurston, William P.},
     TITLE = {Finite covers of random 3-manifolds},
   JOURNAL = {Invent. Math.},
  FJOURNAL = {Inventiones Mathematicae},
    VOLUME = {166},
      YEAR = {2006},
    NUMBER = {3},
     PAGES = {457--521},
      ISSN = {0020-9910,1432-1297},
   MRCLASS = {57N10 (20F34 20F65 57M10 57M50)},
  MRNUMBER = {2257389},
MRREVIEWER = {Bruno\ P.\ Zimmermann},
       DOI = {10.1007/s00222-006-0001-6},
       URL = {https://doi.org/10.1007/s00222-006-0001-6},
}

@article {Urbano13_one-sided,
    AUTHOR = {Urbano, Francisco},
     TITLE = {Second variation of one-sided complete minimal surfaces},
   JOURNAL = {Rev. Mat. Iberoam.},
  FJOURNAL = {Revista Matem\'atica Iberoamericana},
    VOLUME = {29},
      YEAR = {2013},
    NUMBER = {2},
     PAGES = {479--494},
      ISSN = {0213-2230,2235-0616},
   MRCLASS = {53A10 (53C42)},
  MRNUMBER = {3047425},
MRREVIEWER = {Makoto\ Sakaki},
       DOI = {10.4171/RMI/727},
       URL = {https://doi.org/10.4171/RMI/727},
}

@article {HershkovitsWhite23_Avoid,
    AUTHOR = {Hershkovits, Or and White, Brian},
     TITLE = {Avoidance for set-theoretic solutions of mean-curvature-type
              flows},
   JOURNAL = {Comm. Anal. Geom.},
  FJOURNAL = {Communications in Analysis and Geometry},
    VOLUME = {31},
      YEAR = {2023},
    NUMBER = {1},
     PAGES = {31--67},
      ISSN = {1019-8385,1944-9992},
   MRCLASS = {53E10},
  MRNUMBER = {4652509},
}

@article {Maher10_RandomHeegaard,
    AUTHOR = {Maher, Joseph},
     TITLE = {Random {H}eegaard splittings},
   JOURNAL = {J. Topol.},
  FJOURNAL = {Journal of Topology},
    VOLUME = {3},
      YEAR = {2010},
    NUMBER = {4},
     PAGES = {997--1025},
      ISSN = {1753-8416,1753-8424},
   MRCLASS = {37E30 (20F65 57M50 57N10 60G50)},
  MRNUMBER = {2746344},
MRREVIEWER = {Georgios\ Tsapogas},
       DOI = {10.1112/jtopol/jtq031},
       URL = {https://doi.org/10.1112/jtopol/jtq031},
}

\end{document}